\documentclass[12pt]{article}

\usepackage{url}
\usepackage{mathtools}
\usepackage{amssymb}
\usepackage{amsthm}
\usepackage{empheq}
\usepackage{latexsym}
\usepackage{enumitem}
\usepackage{eurosym}
\usepackage{dsfont}
\usepackage{appendix}
\usepackage{color} 
\usepackage[unicode]{hyperref}
\usepackage{frcursive}
\usepackage[utf8]{inputenc}
\usepackage[T1]{fontenc}
\usepackage{geometry}
\usepackage{multirow}
\usepackage[colorinlistoftodos]{todonotes}
\usepackage{lmodern}
\usepackage{anyfontsize}
\usepackage{stmaryrd}
\usepackage{natbib}
\usepackage{cleveref}

\usepackage{amsbsy}

\usepackage{cancel}

\usepackage{amsmath}

\setcitestyle{numbers,open={[},close={]}}

\definecolor{red}{rgb}{0.7,0.15,0.15}
\definecolor{green}{rgb}{0,0.5,0}
\definecolor{blue}{rgb}{0,0,0.7}
\hypersetup{colorlinks, linkcolor={red},citecolor={green}, urlcolor={blue}}
			
\makeatletter \@addtoreset{equation}{section}

\newtheorem{theorem}{Theorem}[section]

\newtheorem{lemma}[theorem]{Lemma}
\newtheorem{proposition}[theorem]{Proposition}

\newtheorem{remark}[theorem]{Remark}

\def\beq{\begin{eqnarray}}
\def\eeq{\end{eqnarray}}
\def\be*{\begin{eqnarray*}}
\def\ee*{\end{eqnarray*}}

\def \C{\mathbb{C}}

\def \E{\mathbb{E}}

\def \N{\mathbb{N}}

\def \P{\mathbb{P}}

\def \R{\mathbb{R}}

\def\Fc{{\cal F}}

\def\Hc{{\cal H}}

\def\Lc{{\cal L}}

\def\Tc{{\cal T}}

\def\0{\mathbf{0}}

\def\normeL2#1{\left\|{#1}\right\|_{L^2}}

\def \Prod{\displaystyle\prod}
\def \Int{\displaystyle\int}
\def \Frac{\displaystyle\frac}
\def \Inf{\displaystyle\inf}

\def\trace{{\rm Tr}}

\def \1{\mathds{1}}
\def \d{{\rm d}}
\def \i{{\rm i}}

\def \Re{\mathrm{Re}}

\def\restrict#1{\raise-.5ex\hbox{\ensuremath|}_{#1}}

\DeclareUnicodeCharacter{014D}{\=o}
 \title{Infinite Dimensional Mean-Field Belavkin Equation: Well-posedness and Derivation}

\author{
 Anne de Bouard\footnote{CMAP CNRS, Ecole Polytechnique, I. P. Paris, anne.debouard@polytechnique.edu.} \and Gaoyue Guo\thanks{Universit\'e Paris--Saclay  CentraleSup\'elec, MICS and CNRS FR-3487, gaoyue.guo@centralesupelec.fr. }
    \and Théo Hérouard\thanks{CMAP CNRS, Ecole Polytechnique, I. P. Paris,
        theo.herouard@polytechnique.edu.  }
    }
             \date{\today}

\begin{document}

\maketitle
 
\begin{abstract}
We analyze the mean-field limit of a stochastic Schrödinger equation arising in quantum optimal control and mean-field games, where 
$N$ interacting particles undergo continuous indirect measurement. For the open quantum system described by Belavkin's filtering equation, we derive a mean-field approximation under minimal assumptions, extending prior results limited to bounded operators and finite-dimensional settings. By establishing global well-posedness via fixed-point methods—avoiding measure-change techniques—we obtain higher regularity solutions. Furthermore, we prove rigorous convergence to the mean-field limit in an infinite-dimensional framework. Our work provides the first derivation of such limits for wave functions in 
$L^2(\R^d)$, with implications for simulating and controlling large quantum systems.
\end{abstract}

\section{Introduction}

This paper presents a mathematical analysis and derivation of a mean-field equation arising in quantum optimal control and quantum mean-field games. We consider a system of $N$
interacting quantum particles in $\R^d$, constructed on the Hilbert space $L^2(\R^{dN}; \C)$. 

In the absence of external interactions, the system evolves under the Hamiltonian $H_N$
\begin{equation*}
    H_N := -\sum_{j=1}^N \Delta_{x_j} + \Frac{1}{N} \sum_{1 \leq i < j \leq N} V(x_i - x_j),
\end{equation*}
where $x_1, \ldots, x_N \in \R^d$ refer to the positions of the $N$ particles. In the Hamiltonian, $V$ is a two-body interaction potential independent of $N$. The dynamics of the system is then governed by the $N$-particles Schrödinger equation:
\begin{equation*}
    \i \partial_t  \psi_{N}(t) = H_N \psi_{N}(t),
\end{equation*}
for an initial condition $\psi_{N,0} \in L^2(\R^{Nd}; \C)$.

In quantum optimal control, continuous measurements introduce feedback but risk perturbing the system (the quantum Zeno effect). To mitigate this, indirect measurements are employed, leading to an open system described by Belavkin's quantum filtering equation due to Belavkin \cite{belavkin1992quantum}, which can be derived from quantum probability (see for instance \cite{bouten2007introduction}). For pure states, this reduces to a stochastic wave function evolution:
\begin{equation}
\label{eq:Npeq}
    \begin{aligned}
    \d \psi_{N}(t) &= -\i H_N \psi_{N}(t) \d t - \frac{1}{2}\sum_{j=1}^N(L_j^* L_j -2 \langle L_j \rangle_{\psi_{N}(t)} L_j +\langle L_j \rangle_{\psi_{N}(t)}^2) \psi_{N}(t)\d t\\
    &+ \sum_{j=1}^N(L_j -  \langle L_j \rangle_{\psi_{N}(t)})\psi_{N}(t) \d B^j_t.
    \end{aligned}
\end{equation}
Here, the operators $(L_j)_{j\leq N}$ model the interaction of the system with the environment, $L^*$ is the adjoint of the operator $L$, $\langle L_j \rangle_{\psi_{N}(t)}$ is the average value of $L_j$ in state $\psi_{N}(t)$ (where the rigorous definition is given by \eqref{eq:defaverage}) and $(B^j)_{j \leq N}$ is a sequence of independent Brownian motions. The subscript $j$ on the operator $L_j$ indicates that the operator only acts on the $j$-th particle.

For large $N$, such equations are very expensive to simulate. The scaling in $N^{-1}$ in front of the interaction potential characterizes mean-field models. With such a scaling, it is expected that the potential interaction becomes independent of the other particles when the number of particles tends towards infinity: this is the mean-field approximation. This approximation states that if $\psi_{N,0}$ is factorized, then the factorization is asymptotically preserved by the evolution. Having a mean field approximation can, for example, greatly reduce the cost of simulating or controlling large many body systems.

For closed quantum systems, a vast literature exists to show that a mean-field approximation is possible. Formally, for $N$ large, one can write $\psi_{N}(t) \simeq \varphi(t)^{\otimes N}$ where $\varphi$ is a solution of the Hartree equation
\begin{equation*}
    \i \partial_t \varphi(t) = - \Delta_x \varphi(t) + (V \star |\varphi(t)|^2 )\varphi(t).
\end{equation*}
Numerous results on the study of such nonlinear Schrödinger equations are presented in \cite{cazenave2003semilinear}.

This approximation result was first rigorously demonstrated by Hepp \cite{Hepp1974} under restrictive conditions on $V$. Now there are a large number of methods for proving such result, for more general interaction potential, and numerous reviews on those results (for example \cite{golse2016dynamics, rougerie2021scaling, benedikter2016effective}). For instance we can mention the method using the BBGKY hierarchy \cite{spohn1980kinetic, bardos2000weak, erdos2001derivation}, the coherent state method \cite{Hepp1974, rodnianski2009quantum} or the method developed by Pickl which consists on "Gronwalling" a good indicator of convergence \cite{pickl2011simple, knowles2010mean}.

For open quantum systems, in order to obtain a mean-field limit equation, a first assumption needs to be done on the coupling operators. The $N$-coupling operators will be linked to a similar operator $L$ acting on elements of $L^2(\R^d)$, and $L_j$ stands for the application of $L$ to the $j$-th particle. Kolokolstov, in \cite{kolokoltsov2021law, kolokoltsov2022quantum}, first derived the mean-field equation using Pickl's method in the case where the bosons are described in a finite dimensional Hilbert space, with a bounded Hamiltonian. The mean-field equation is
\begin{equation}
    \label{eq:mfeqintro}
    \begin{aligned}
     \d \varphi(t) &= -\i (-\Delta + V \star \E |\varphi(t)|^2 ) \varphi(t)\d t - \frac{1}{2}(L^* L -2 \langle L \rangle_{\varphi(t)} L  +\langle L \rangle_{\varphi(t)}^2)\varphi(t)\d t\\
     &+ (L - \langle L \rangle_{\varphi(t)})\varphi(t) \d \beta_t,
    \end{aligned}
\end{equation}
where $\beta$ is a Brownian motion.

In \cite{kolokoltsov2021law}, the derivation was done rigorously for a bounded Hamiltonian, a bounded operator $L$ and an interaction given by a Hilbert-Schmidt operator. Then, in \cite{kolokoltsov2022quantum} the derivation was extended for a bounded interaction potential with a bound useful only in a finite dimensional setting, i.e. when the state space $X$ is finite. Indeed, the final bound used for the convergence contains the following term:
\begin{equation*}
    \int_{X}\sqrt{\E[|\phi(x)|^4]}\d x,
\end{equation*}
where $\phi$ is the solution of the mean-field equation and $X$ is the state space. If such a bound can be easily obtained for a finite $X$, it must be studied in greater depth in the case of $\R^d$. In a similar finite dimensional setting, the control over Pickl's quantity has been obtained from Belavkin equation for density matrices \cite{chalal2023mean}. In \cite{kolokoltsov2022quantum, chalal2023mean} the main focus was to use the limiting equation for quantum mean-field game. The master equation for density matrices was first studied for an unbounded Hamiltonian and bounded interaction in \cite{kolokoltsov2025quantum}, with a proof inspired by \cite{mora2008basic}.

One of the main assumptions of those articles was the boundedness of the coupling operator $L$. For unbounded operator, the existence and uniqueness of solutions to the Belavkin equation is more involved and requires additional hypotheses (see for instance \cite{mora2008basic, kolokoltsov2025mathematicaltheoryquantumstochastic}).

Our purpose in the present paper is to study in depth equation \eqref{eq:mfeqintro}, from global existence to its derivation. The focus is on the dynamics of the wave function, which has more non-linearity than the Belavkin equation. We prove the global existence with some fixed point arguments, and without the usual change of measure and renormalisation method. This enables us to find solutions with a higher regularity under additional assumptions. By improving regularity, we are able to provide a bound for the mean-field limit that is better adapted to the infinite dimensional framework. Therefore, to the best of our knowledge, this paper is the first to rigorously prove the mean-field convergence of a quantum particle system with noise in the infinite dimensional case with an interaction potential in an unbounded domain.

\section{Preliminaries and main results}

Let $d$ denote the dimension of the environment. We define the usual Hilbert space $L^2(\R^d;\C)$ (denoted in the following only by $L^2(\R^d)$) of square integrable functions equipped with the symmetric scalar product
\begin{equation*}
    (u,v)_{L^2_x} = \Re \int_{\R^d} u(x) \overline{v}(x) \d x,
\end{equation*}
and denote $| \cdot |_{L^2_x}$ the associated norm. An element $\psi \in L^2(\R^d)$ will be called a wave function if it has unitary $L^2$-norm. For the study of the equation, we also define the following spaces for $t>0$:
\begin{equation*}
    C_t L^2 := C([0,t], L^2_x),
\end{equation*}
equipped with the norm:
\begin{equation*}
    |u|_{C_t L^2} := \sup_{s \in [0,t]} |u(s)|_{L^2_x}.
\end{equation*}
For $p \geq 1$, we define the space of complex valued $L^p$ functions in the following way
\begin{equation*}
    L^p(\R^d) := \{u: \R^d \to \C; \int_{\R^d}|u(x)|^p \d x < + \infty \},
\end{equation*}
with its usual norm. If the functions are real valued, we denote the space by $L^p(\R^d; \R)$. We also define the usual Sobolev spaces $W^{1,p}(\R^d)$ for $p \geq 1$ by
\begin{equation*}
    W^{1, p}(\R^d) = \{u \in L^p(\R^d); \nabla u \in L^p(\R^d; \C^d)\}.
\end{equation*}
When $p=2$, set $W^{1, 2}(\R^d) = H^1(\R^d)$ which is a Hilbert space with scalar product
\begin{equation*}
    (u,v)_{H^1_x} := (u,v)_{L^2_x} + (\nabla u, \nabla v)_{L^2_x},
\end{equation*}
and its associated norm denoted $| \cdot |_{H^1_x}$. We denote the corresponding spaces of continuous functions with values in $H^1$ for $t > 0$:
\begin{equation*}
    C_t H^1 := C([0,t], H^1_x),
\end{equation*}
equipped with the norm:
\begin{equation*}
    |u|_{C_t H^1} := \sup_{s \in [0,t]} |u(s)|_{H^1_x}.
\end{equation*}

For a bounded operator $L$ on $L^2(\R^d)$ and $\psi \in L^2(\R^d)$, we define the following function:
\begin{equation}
    \label{eq:defaverage}
    \langle L \rangle_{\psi} := (\psi, L \psi)_{L^2_x}.
\end{equation}
Note that, if $\psi$ is a wave function, then the above expression represents the average of $L$ in state $\psi$.

For $(A, |\cdot|_A)$ and $(B, |\cdot|_B)$ two Banach spaces, we denote by $\Lc(A, B)$ the set of bounded linear operators from $A$ to $B$. The usual operator norm is defined by
\begin{equation*}
    \| L \|_{\Lc(A,B)} := \sup_{|x|_A =1} |L x|_B.
\end{equation*}
When the operators are endomorphism in $A$, we denote the space of bounded operators by $\Lc(A)$ and its usual norm by $\| \cdot \|_{\Lc(A)}$. Moreover, we define the set of Hilbert-Schmidt operator in $A$ by $\Lc^2(A)$. This space is a Hilbert space with the following scalar product
\begin{equation*}
    (A,B)_{\Lc^2} := \trace A^* B.
\end{equation*}
We denote $\Lc^1(A)$ the set of trace class linear operators in $A$. We recall that the trace norm is given, for $L \in \Lc^1(A)$, by
\begin{equation*}
    \| L \|_{\Lc^1(A)} =   \trace |L|,
\end{equation*}
where $|L|$ stands for the modulus of $L$ (see \cite{kukush2019} for basic definitions and properties). If there is no ambiguity on the space $A$, we will note the operator norm $\| \cdot \|$, the Hilbert-Schmidt norm $\| \cdot \|_2$ and the trace norm $\| \cdot \|_1$. These three norms are linked by the following relationship, for $\rho \in \Lc(A)$
\begin{equation}
    \label{eq:inegnormop}
    \| \rho \| \leq \| \rho \|_2 \leq \| \rho \|_1.
\end{equation}

For an operator $L$ in $\Lc^2(L^2(\R^d))$, there exists a unique function $l \in L^2(\R^d \times \R^d)$ such that, for all $f \in L^2(\R^d)$
\begin{equation*}
    L f(\cdot) = \int_{\R^d} l(\cdot, y) f(y) \d y.
\end{equation*}
The function $l$ is called the integral kernel of the operator $L$.

Let $V\in L^{\infty}(\R^d)$ be an even interaction potential, and consider an unbounded free Hamiltonian $H : D(H) \subset L^2(\R^d) \rightarrow L^2(\R^d)$. Let $L$, a bounded operator on $L^2(\R^d)$, be the coupling operator. For the probabilistic framework, we consider a probability space $(\Omega, \Fc, \P)$, endowed with a filtration $(\Fc_t)_{t \geq 0}$. Let $(\beta_t)_{t\geq0}$ be a Brownian motion defined on the stochastic basis $(\Omega, \Fc, (\Fc_t)_{t \geq 0}, \P)$.  The general form of the mean-field equation is then
\begin{equation}
    \label{eq:mfeq}
    \begin{aligned}
     \d u(t) &= \left(-\i (H + V \star \E |u(t)|^2 ) - \frac{1}{2}(L^* L -2 \langle L \rangle_{u(t)} L +\langle L \rangle_{u(t)}^2)\right)u(t)\d t \\
&+ (L - \langle L \rangle_{u(t)})u(t) \d \beta_t.
    \end{aligned}
\end{equation}
The unknown solution $u$ of \eqref{eq:mfeq} is a complex random field defined on the stochastic basis $(\Omega, \Fc, (\Fc_t)_{t \geq 0}, \P)$. We first establish the global existence of the solution to \eqref{eq:mfeq}.

\begin{theorem}
\label{thm:meq}
Fix $T_0>0$. Assume that $H : D(H) \subset L^2(\R^d) \to L^2(\R^d)$ is a self-adjoint operator, that $L \in \Lc(L^2(\R^d))$ and that the potential $V \in L^{\infty}(\R^d)$ is even. Moreover, assume that $u_0$ is a $\Fc_0$-measurable random variable such that $|u_0|_{L^2} = 1$ almost surely. Then there exists a unique solution $u \in L^2(\Omega, C([0,T_0]; L^2(\R^d)))$ of \eqref{eq:mfeq} such that $u(0) = u_0$. Moreover, almost surely:
\begin{equation*}
    |u(t)|_{L^2_x} = 1 \text{, for all } t \in [0, T_0].
\end{equation*}
\end{theorem}

Then, we are interested in the convergence of the $N$-particles dynamics to the mean-field dynamics. We will restrict ourselves to the case where $d=3$ and a specific Hamiltonian $H = - \Delta$. For the convergence, we will need more regular solutions, which means adding new assumptions on the coupling operator $L$. For that, we will prove a bound that shows the preservation of the $H^1(\R^3)$ regularity for the solution of \eqref{eq:mfeq}.

\begin{theorem}
\label{thm:meqh1}
Fix $T_0>0$. Assume that $H = - \Delta$, that the potential $V \in L^{\infty}(\R^3)$ and that $u_0 \in L^2(\Omega, H^1(\R^3))$ is a $\Fc_0$-measurable random variable such that $|u_0|_{L^2} = 1$ almost surely. Moreover, assume that the operator $L\in \Lc(L^2_x)$ satisfies $[\nabla, L ] \in \Lc( H^1_x, L^2_x)$ and $[\nabla, L^* L ] \in \Lc( H^1_x, L^2_x)$. Then the unique solution $u$ of \eqref{eq:mfeq} such  that $u(0) = u_0$ given by Theorem \ref{thm:meq} has continuous $H^1$-valued paths, i.e. $u \in L^2(\Omega, C([0,T_0]; H^1(\R^3)))$.
\end{theorem}

Before stating the convergence theorem, we specify the type of convergence required, which needs the introduction of density operators. A trace class operator is called a density operator if it is symmetric, positive and of trace one. Let $\psi_N$ be a solution of the $N$-body dynamics given by \eqref{eq:Npeq}. Now, let us define the associated density operator $\rho^N_t$ by its integral kernel 
\begin{equation}
    \label{eq:defrhon}
    \rho^N_t(X_N, Y_N) = \psi_N(X_N, t) \overline{\psi_N(Y_N, t)},
\end{equation}
for $X_N, Y_N \in \R^{dN}$. One can easily show that with this definition, $\rho^N_t$ is a symmetric positive trace class operator. Since the dynamics of $\psi_N$ is given by \eqref{eq:Npeq}, the evolution of $\rho^N$ is then governed by the so-called Belavkin equation:
\begin{equation}\label{eq:BelavkinNp}
    \begin{aligned}
     \mathrm{d} \rho_t^{N} &= -\i{[H_{N},\rho_{t}^{N}]}\mathrm{d}t
+ \sum_{j=1}^N\Big({L}_j\rho_{t}^{N}{L}^{*}_{j} - \frac{1}{2}\big\{{{ L^*_j}}{ L}_{j},\rho_{t}^{N}\big\}\Big)\mathrm{d}t \\
&+\sum_{j=1}^N \Big({ L}_j\rho_{t}^{N}
+ \rho_{t}^{N}{L}^{*}_j -tr\big(({L}_j + {L}_j^{*})\rho_{t}^{N}\big)\rho_{t}^{N}\Big)\mathrm{d}B_t^{j}.
    \end{aligned}
\end{equation}
Moreover, denote by $\phi^{MF,j}$ a solution to the mean-field equation \eqref{eq:mfeq} driven by the Brownian motion $B^j$. For such solution, we define the projector operator in the following way, for $t \geq 0$
\begin{equation*}
    p_{j,t} := |\phi^{MF,j}_t \rangle \langle \phi^{MF,j}_t| = ( \phi^{MF,j}_t, \cdot)_{L^2_x}\phi^{MF,j}_t.
\end{equation*}
The remaining part of this article aims to show that, in a certain sense, the mean-field approximation holds for $N$ large:
\begin{equation*}
    \rho^N_t \simeq \bigotimes_{j=1}^N p_{j,t}.
\end{equation*}
However, the convergence in $\Lc^2$-norm cannot be expected and thus a weaker notion of convergence needs to be introduced. For $J \subset \{1, \ldots, N \}$ and $X_N \in R^{dN}$, we define $X_J = (x_j)_{j\in J} \in \R^{|J|d}$ and $\hat{X}^J_N$ the variables of the particles outside $J$. Then, for an operator $\rho \in \Lc^1(L^2(\R^{d N}))$ we introduce the partial trace over the variables outside $J$ with its integral kernel by
\begin{equation*}
    \big( \trace_{\{1, \ldots, N \} \backslash J} \rho \big)(X_J, Y_J) := \int_{\R^{d(N-|J|)}} \rho(X^N, Y^N) \mathds{1}_{ \{\hat{Y}^N_J =\hat{X}^N_J \} } \d \hat{X}^N_J,
\end{equation*}
where $X_J, Y_J \in R^{d |J|}$. From this definition, the partial trace of a density operator $\rho^N$ with respect to all the variables outside $J$ is given by
\begin{equation*}
    \rho^{N,J}_t := \trace_{\{1, \ldots, N \} \backslash J} \rho^N_t.
\end{equation*}
This is the so-called reduced marginal operator of the particles $(x_j)_{j \in J}$. Then, for $J $ subset of $\{1, \ldots, N \}$, we can define the trace norm distance estimate for this marginal by 
\begin{equation*}
    R_{N,J}(t) := \E \|\rho^{N,J}_t - \bigotimes_{j \in J} p_{j,t}\|_{\Lc^1(L^2(\R^{d |J|}))}.
\end{equation*}
To study those indicators, we will follow the proof of \cite{kolokoltsov2022quantum}, itself inspired from the method of Pickl for the mean-field convergence in the case of a deterministic closed system. We can then give a rigorous sense to the mean-field approximation.
\begin{theorem}
\label{thm:cvgmf} Let $T_0>0$ and $p \geq 4$ and assume that $H,L$ and $V$ satisfie the assumptions of Theorem \ref{thm:meqh1}. Let $(B^j)_{J \in \N}$ be a sequence of independent Brownian motions on the underlying stochastic basis. Let $\phi_0 \in L^p(\Omega, H^1(\R^d))$ be a $\Fc_0$-measurable random variable such that $|\phi_0|_{L^2_x} = 1$ almost surely. For $N \geq 2$, denote by $\Psi_N$ the solution of \eqref{eq:Npeq} with $\Psi_N(0) = \phi_0^{\otimes N}$ and $\rho^N$ its associated density operator (given by \eqref{eq:defrhon}). Finally, for $j \in \N$, denote by $\phi^{MF,j}$ the solution of \eqref{eq:mfeq} driven by the Brownian motion $B^j$ with $\phi^{MF,j}(0) = \phi_0$. Moreover suppose that $V$ vanishes at infinity, i.e.
\begin{equation*}
    \lim_{|x| \rightarrow + \infty} V(x) = 0.
\end{equation*}
Then, for any $J$ finite subset of $\N$, 
\begin{equation}
    \label{eq:cvgtrp}
    \lim_{N \rightarrow + \infty} \sup_{t \in [0, T_0]} R^{N, J}(t) = 0.
\end{equation}
In particular, for the first marginal
\begin{equation}
    \label{eq:cvgpmarg}
    \lim_{N \rightarrow + \infty} \sup_{t \in [0, T_0]} \E \|\rho^{N,1}_t - |\phi^{MF,1}_t \rangle \langle \phi^{MF,1}_t| \|_1 = 0.
\end{equation}
\end{theorem}
The limit given in \eqref{eq:cvgpmarg} states that, when $N$ goes to infinity, the behaviour of one particle can be approximated by the solution of the mean-field equation \eqref{eq:mfeq}.

\begin{remark}
The well-posedness of equation \eqref{eq:Npeq} is not explicitly written in the article. However, using similar arguments as in the proof of Proposition \ref{prop:meq-interm} with $\xi = 0$ and a finite number of Brownian motions, one can prove the global well-posedness of \eqref{eq:Npeq}. The almost-sure preservation for all time of the $L^2$-norm also follows in a similar fashion.
\end{remark}

This paper is organized as follows. In section \ref{sec:mfeq}, we will prove the global  well-posedness of \eqref{eq:mfeq}. For that, we introduce an intermediate equation where the Hartree non-linearity is decoupled from the rest of the equation, which is then studied with a truncation procedure. The proof of Theorem \ref{thm:meq} will follow, by a fixed point procedure on the intermediate equation. In section \ref{sec:propreg} we first discuss the existence of solutions to \eqref{eq:mfeq} in $H^1(\R^3)$ under more restrictive assumptions. Second, moments estimates on the $H^1$-norm of the solution will be obtained. Finally in section \ref{sec:mflim}, we will prove Theorem \ref{thm:cvgmf} under these assumptions, which is the rigorous derivation of the limiting equation from a many body system, using Pickl's convergence indicator. In the appendix, we prove some technical results needed in the course of the proofs.

\section{Mean-field equation}\label{sec:mfeq}

This section is devoted to prove Theorem \ref{thm:meq}. We fix a stochastic basis $(\Omega, \Fc, (\Fc_t)_{t \geq 0}, \P)$ and a Brownian motion $\beta$ adapted to the filtration $(\Fc_t)_{t \geq 0}$.

To demonstrate the theorem, we will start by studying the following intermediate equation
\begin{equation}
\label{eq:interm}
    \begin{aligned}
     \d u(t) &= -\i (H + V \star \xi_t ) u(t) \d t - \frac{1}{2}\left( L ^* L -2  \langle L \rangle_{u(t)} L + \langle L \rangle_{u(t)}^2 \right) u(t)\d t \\ 
     &+(L -  \langle L \rangle_{u(t)})u(t) \d \beta_t,
    \end{aligned}
\end{equation}
where $\xi$ is a deterministic function which has sufficient regularity.

\begin{proposition}
\label{prop:meq-interm}
Fix $T_0>0$ and assume that $H, L, V$ and $u_0$ satisfy the assumptions of Theorem \ref{thm:meq}. If further $\xi \in C([0,T], L^1(\R^d, \R))$ then, there exists a unique solution $u \in L^2(\Omega, C([0,T_0]; L^2(\R^d)))$ of \eqref{eq:interm} such that $u(0) = u_0 $. Moreover, almost surely, 
\begin{equation*}
    |u(t)|_{L^2_x} = 1 \text{, for all } t \in [0, T_0].
\end{equation*}
\end{proposition}

With this proposition in hand, we will perform a fixed point argument on $\xi$ to prove the global existence result for \eqref{eq:mfeq}.

\subsection{A truncated equation}\label{subsection:truncequation}

To show the existence and the uniqueness of the solution of equation \eqref{eq:interm}, it will be convenient to study a truncated equation. For that, we take a function $\theta \in C^{\infty}(\R_+, \R_+)$ with compact support, bounded by $1$ and such that
\begin{equation*}
    \theta(x) = 0, \text{ for } x\geq 2 \text{ and } \theta(x) = 1, \text{ for } x\leq 1.
\end{equation*}
Fix $R >0$, we define the following truncation function by:
\begin{equation*}
    \Theta_R(u\restrict{[0,t]}) = \theta\left(\Frac{|u|_{C_t L^2}}{R} \right).
\end{equation*}
To simplify notations, we define the following operator-valued functions, for $X \in L^2(\R^d)$:
\begin{equation}
\label{eq:defF1}
    F_1(X) = L ^* L -2 \langle L \rangle_{X} L +\langle L \rangle_{X}^2 I_d,
\end{equation}
and 
\begin{equation}
    \label{eq:defF2}
    F_2(X) = L - \langle L \rangle_{X}I_d.
\end{equation}
Then, using $F_1$ and $F_2$, we can define the following truncated equation
\begin{equation}
\label{eq:trunc}
    \begin{aligned}
    \d u^R(t) &= -\i (H + V \star \xi_t ) u^R(t) \d t 
    - \frac{1}{2}\Theta_R(u^R\restrict{[0,t]})^2 F_2(u^R(t))u^R(t)\d t\\
    &+ \Theta_R(u^R\restrict{[0,t]}) F_2(u^R(t))u^R(t) \d \beta_t,
    \end{aligned}
\end{equation}
where $\xi \in C([0,T], L^1(\R^d, \R))$.

\begin{proposition}\label{prop:eqtrunc}
Let $T_0 >0$ and $R\geq 1$. Assume that $H : D(H) \subset L^2(\R^d) \to L^2(\R^d)$ is an unbounded self-adjoint operator, that $L \in \Lc(L^2(\R^d))$ and that $V \in L^{\infty}(\R^d)$, the potential, is even. Moreover, assume that $u_0 \in L^2(\Omega, L^2(\R^d))$ is a $\Fc_0$-measurable random variable and that $\xi \in C([0,T], L^1(\R^d, \R))$. Then there exists a unique solution $u^R$ in $L^2(\Omega, C([0,T_0], L^2(\R^d)))$ to \eqref{eq:trunc} such that $u(0) = u_0$.
\end{proposition}

The proof of this proposition relies on a fixed point argument, in a similar way as in \cite{debouard2003snls}. Before detailing the proof, notice the following inequality, which is a direct consequence of Cauchy-Schwarz inequality, for $X, Y \in L^2(\R^d)$
\begin{equation}
    \label{eq:basicestim}
    |(X, L Y)_{L^2_x}| \leq \| L \| |X|_{L^2_x} |Y|_{L^2_x}.
\end{equation}
From this basic estimate and the definition \eqref{eq:defaverage}, we can state the following lemma.

\begin{lemma}\label{lem:estimateavg}
Fix $R>0$. Let $X, Y \in L^2(\R^d)$, such that $|X|_{L^2_x} \vee |Y|_{L^2_x} \leq 2 R$, $L$ be a bounded operator on $L^2(\R^d)$ and $\Theta_R$ be defined as before. Then:
\begin{equation*}
    |\langle L \rangle_{X}| \leq 4 R^2 \| L \|,
\end{equation*}
\begin{equation*}
    |\langle L \rangle_{X} - \langle L \rangle_{Y}| \leq 4 R \| L \| |X-Y|_{L^2_x}.
\end{equation*}
Moreover the following bound on the truncation function holds, for $T>0$ and $X,Y \in C_TL^2$:
\begin{equation*}
    |\Theta_R(X) - \Theta_R(Y)| \leq \Frac{|\theta'|_{L^{\infty}}}{R}|X-Y|_{C_TL^2}.
\end{equation*}

\end{lemma}

To help during the computations of the fixed point procedure, the following Lemma states key estimates on the two functions $F_1$ and $F_2$. The proof is given in Appendix \ref{sec:appendix-estF}.
\begin{lemma}
\label{lem:estimateF}
Fix $R>0$. Let $X, Y \in L^2(\R^d)$, such that $|X|_{L^2_x} \vee |Y|_{L^2_x} \leq 2 R$ and $L$ a bounded operator. Then the following estimates hold:
\begin{enumerate}
    \item There exists $C_1(R, \| L \|) >0$ such that
    \begin{equation}
    \label{eq:estimatef11}
        |F_1(X)X|_{L^2_x} \leq C_1(R, \| L \|).
    \end{equation}
    \item There exists $C_2(R, \| L \|) > 0$ such that
    \begin{equation}
    \label{eq:estimatef12}
        |F_1(X)X - F_1(Y)Y|_{L^2_x} \leq C_2(R, \| L \|)|X-Y|_{L^2_x}.
    \end{equation}
    \item There exists $C_3(R, \| L \|)>0$ such that
    \begin{equation}
    \label{eq:estimatef21}
        |F_2(X)X|_{L^2_x}\leq C_3(R, \| L \|)
    \end{equation}
    \item There exists $C_4(R, \| L \|)>0$ such that
    \begin{equation}
    \label{eq:estimatef22}
        |F_2(X)X - F_2(Y)Y|_{L^2_x}\leq C_4(R, \| L \|)|X-Y|_{L^2_x}.
    \end{equation}
\end{enumerate}
\end{lemma}
\begin{proof}[Proof of Proposition \ref{prop:eqtrunc}]
First, define the following application, for $u \in C_{T_0}L^2$:

\begin{equation}
    \begin{aligned}
    \Tc_R u(t) &= S(t) u_0 - \i \int_0^t S(t-s) [(V \star \xi_s) u(s)]\d s  \\
    &- \frac{1}{2} \int_0^t  \Theta_R(u\restrict{[0,s]})^2 S(t-s) [F_1(u(s))u(s)]\d s \\
&+ \int_0^t\Theta_R(u\restrict{[0,s]}) S(t-s) [F_2(u(s))u(s)] \d \beta_s,
    \end{aligned}
\end{equation}
where $(S(t))_{t \in \R}$ is the group with infinitesimal generator $-\i H$ given by Stone's theorem, since $H$ is self-adjoint on $L^2(\R^d)$. Using \eqref{eq:estimatef11}, \eqref{eq:estimatef21} and the continuity in time of the group $S$, it is standard that if $u$ has almost surely continuous trajectories in $L^2(\R^d)$ then so does $\Tc_R$, i.e. $\Tc_R$ maps $C([0,T], L^2(\R^d))$ to itself.

We show that $\Tc_R$ is a contraction on $L^2(\Omega, C_T L^2)$ for $T \leq T_0$ small enough. 
Let $u, v \in L^2(\Omega, C_{T} L^2)$, we define the following stopping times $t_R^u$ and $t_R^v$ by:
\begin{equation}
    t_R^u := \inf \{t \leq T, |u|_{C_t L^2} \geq 2R \},
\end{equation}
and similarly $t_R^v$ with $v$ instead of $u$.
Fix $t \leq T_0$ and $\omega \in \Omega$ such that $u(\omega)$ and $v(\omega)$ are continuous in time, then without loss of generality, assume that $t_R^u \leq t_R^v$. Since $S$ is a unitary group,
\begin{equation*}
    \begin{aligned}
    &|\Tc_R u(t) - \Tc_R v(t)|_{L^2_x} \leq \int_0^t |(V \star \xi_s) (u(s)- v(s))|_{L^2_x} \d s\\
&+ \frac{1}{2}\int_0^{t^u_R\wedge t} |\Theta_R(u\restrict{[0,s]})^2||F_1(u(s))u(s) - F_1(v(s))v(s)|_{L^2_x} \d s \\
&+ \frac{1}{2}\int_0^{t^u_R\wedge t} |\Theta_R(u\restrict{[0,s]})^2 - \Theta_R(v\restrict{[0,s]})^2| |F_1(v(s))v(s)|_{L^2_x} \d s \\
&+ \frac{1}{2}\int_{t^u_R \wedge t}^{t^v_R \wedge t} |\Theta_R(v\restrict{[0,s]})^2F_1(v(s))v(s)|_{L^2_x} \d s \\
 &+ |\int_0^t \Theta_R(u\restrict{[0,s]}) S(t-s)F_2(u(s))u(s) - \Theta_R(v\restrict{[0,s]}) S(t-s)F_2(v(s))v(s) \d \beta_s |_{L^2_x} \\
 &= I + \frac{1}{2}II + \frac{1}{2}III + \frac{1}{2}IV + V.
    \end{aligned}
\end{equation*}

We estimate $I$ in the following way:
\begin{eqnarray*}
I &\leq& \int_0^T |V \star \xi_s|_{L^{\infty}_x}|(u(s)- v(s))|_{L^2_x} \d s \\
&\leq & |V|_{L^{\infty}_x}  \int_0^T|\xi_s|_{L^1_x}|(u(s)- v(s))|_{L^2_x} \d s,
\end{eqnarray*}
where we used a convolution inequality. Then, by the definition of the norm of $C_T L^2$:
\begin{equation}
    \label{eq:proof1est1}
    I \leq T |V|_{L^{\infty}_x} |\xi|_{L^{\infty}(0,T; L^1_x)} |u-v|_{C_T L^2}.
\end{equation}
Since $t\in [0, t_R^u]$, we can apply \eqref{eq:estimatef12} to $II$ and obtain
\begin{equation}
\label{eq:proof1est2}
    II \leq T C_2(R) |u-v|_{C_T L^2}.
\end{equation}
To estimate $III$, first notice the following bound
\begin{equation*}
    |\Theta_R(u)^2 - \Theta_R(v)^2| = |\Theta_R(u) - \Theta_R(v)| |\Theta_R(u) + \Theta_R(v)| \leq \Frac{2|\theta'|_{L^{\infty}}}{R}|u-v|_{C_T L^2},
\end{equation*}
where we used Lemma \ref{lem:estimateavg}. Hence, since $t \in [0, t^u_R]$, by \eqref{eq:estimatef11} we have
\begin{equation}
\label{eq:proof1est3}
    III \leq 4 T C_1(R)|\theta'|_{L^{\infty}}|u-v|_{C_T L^2}.
\end{equation}
For $IV$, observe that when $s \in [t^u_R, t^v_R]$, $\Theta_R(u\restrict{[0,s]})$ is null. Thus we have 

\begin{equation}
\label{eq:proof1est4}
    \begin{aligned}
IV &= \int_{t^u_R \wedge t}^{t^v_R \wedge t} |\Theta_R(u\restrict{[0,s]})^2 - \Theta_R(v\restrict{[0,s]})^2||F_1(v(s))v(s)|_{L^2_x} \d s \\
&\leq 4 T C_1(R)|\theta'|_{L^{\infty}} |u-v|_{C_T L^2}.
\end{aligned}
\end{equation}

Finally, to study the martingale part, we recall the following classical result:
\begin{lemma}
\label{lem:itoiso}
Let $\Phi \in L^2(\Omega, L^2(0,T;L^2_x))$ be a predictable process, $S$ be a unitary group and $\beta$ be a Brownian motion, then for all $t \leq T$
\begin{equation*}
    \E \left[|\int_0^t S(t-s)\Phi_s \d \beta_s|^2_{L^2_x} \right] = \E \left[|\int_0^t S(-s)\Phi_s \d \beta_s|^2_{L^2_x} \right]
\end{equation*}
and $(\int_0^t S(-s)\Phi_s \d \beta_s)_{s \leq T}$ is a square integrable martingale. Moreover, by Doob's inequality and Ito's isometry
\begin{equation}
    \label{eq:appliBDG}
    \E \left[\sup_{t\leq T} |\int_0^t S(t-s)\Phi_s \d \beta_s|^2_{L^2_x} \right] \leq 4 \E \left[ \int_0^T |\Phi_s|^2_{L^2_x} \d s\right].
\end{equation}
\end{lemma}

We aim to apply this lemma to the martingale part, i.e with the following predictable process, for $s \in [0, T_0]$:
\begin{equation*}
    \Phi_s = \Theta_R(u\restrict{[0,s]})F_2(u(s))u(s) - \Theta_R(v\restrict{[0,s]})F_2(v(s))v(s).
\end{equation*}
By using the same subdivision of the time interval as the one used for the drift part, and then applying \eqref{eq:estimatef21} and \eqref{eq:estimatef22}, we find that there exists $C_5 >0$ that depends only of $\| L \|$ and $R$ such that
\begin{equation*}
    \int_0^t |\Phi_s|_{L^2_x}^2 \d s \leq C_5 |u-v|^2_{C_T L^2}.
\end{equation*}
Therefore, by applying \eqref{eq:appliBDG} from Lemma \ref{lem:itoiso}:
\begin{equation}
\label{eq:proof1est5}
\begin{aligned}
 &\E \left[ \left|\int_0^{\cdot} S(\cdot - s) [\Theta_R(u\restrict{[0,s]}) F_2(u(s))u(s) - \Theta_R(v\restrict{[0,s]}) F_1(v(s))v(s)]\d \beta_s \right|_{C_T L^2}^2\right] \\
 &\leq 4 C_5 T \E[|u-v|^2_{C_T L^2}].
\end{aligned}
\end{equation}
Thus, collecting \eqref{eq:proof1est1}-\eqref{eq:proof1est4} and \eqref{eq:proof1est5}, there exists an explicit constant $C_6 >0$, which depends only on $R$, $T_0$, $\|L\|$ and $|V|_{L^{\infty}_x}$ such that, for $T \leq T_0$
\begin{equation*}
    \E[|\Tc_R u - \Tc_R v|_{C_T L^2}^2] \leq C_6 T \E[|u-v|^2_{C_T L^2}].
\end{equation*}
This inequality can be written as follows
\begin{equation}
\label{eq:eqtruncfin}
    |\Tc_R u - \Tc_R v|_{L^2(\Omega, C_T L^2)} \leq \sqrt{C_6} T^{1/2}|u-v|_{L^2(\Omega, C_T L^2)}.
\end{equation}
This shows that, by taking $T$ such that
\begin{equation*}
    T \leq \Inf \left(T_0, \Frac{1}{2 C_6(R, T_0, \| L \|, |V|_{L^{\infty}_x})}\right),
\end{equation*}
$\Tc_R$ is a contraction mapping in $L^2(\Omega, C_T L^2)$. Moreover, since $u_0 \in L^2(\Omega, C_T L^2)$, by the same manipulations as before we have that $\Tc_R$ maps $L^2(\Omega, C_T L^2)$ into itself for each $T>0$. Hence, $\Tc_R$ has a pathwise unique fixed point in $C_T L^2$, which is the solution to the truncated equation.

Since the choice of $T$ is independent of $u_0$, the time interval $[0,T_0]$ can be written as
\begin{equation*}
    [0, T_0] = \bigcup_{k=0}^{N-1} [kT, (k+1)T] \cup [NT, T_0],
\end{equation*}
where $N = \lfloor T_0/T \rfloor$. With this writing, the result can easily be extended to the whole interval $[0, T_0]$ by pathwise uniqueness.
\end{proof}

\subsection{Proofs of well-posedness}

In this subsection we are interested in the proof of Theorem \ref{thm:meq}. We start by removing the truncation and proving Proposition \ref{prop:meq-interm}. Then, we are finally equipped to demonstrate Theorem \ref{thm:meq}.

\subsubsection{Proof of Proposition \ref{prop:meq-interm}}

We want to show that for $R$ large enough, the solution of the truncated equation is in fact the solution of equation \eqref{eq:mfeq}.
Let $u_0$ be a $\Fc_0$-measurable random variable with $|u_0|_{L^2_x} = 1$ $ \P$-almost surely and $R \geq 1$, then denote by $u^R$ the unique solution of \eqref{eq:trunc} given by Proposition \ref{prop:eqtrunc}. We first show the conservation of the $L^2$-norm of the solution.

\begin{lemma}\label{lem:estimnormL2}
For all $t \leq T_0$, if we assume $|u_0|_{L^2_x} = 1$ almost-surely, then almost-surely, it holds that
\begin{equation*}
    |u^R(t)|_{L^2_x} = 1 \text{, for all } t \in [0, T_0].
\end{equation*}
\end{lemma}
\begin{proof}

Note that for a self-adjoint operator $A : D(A) \subset L^2_x \mapsto L^2_x$, we have for all $x \in D(A)$
\begin{equation*}
    (x,\i A x)_{L^2_x} = 0.
\end{equation*}
Then, since $u^R$ is a solution of \eqref{eq:trunc}, by Ito's formula
\begin{eqnarray*}
\d |u^R(t)|^2_{L^2_x} &=& 2(u^R(t), -\i (H + V \star \xi_t) u^R(t) )_{L^2_x} \d t \\
&-&  (u^R(t), \Theta_R(u^R\restrict{[0,t]})^2 F_1(u^R(t)) u^R(t) )_{L^2_x} \d t \\
&+& |\Theta_R(u^R\restrict{[0,t]}) F_2(u^R(t))u^R(t)|^2_{L^2_x}\d t \\
&+&(u^R(t), \Theta_R(u^R\restrict{[0,t]}) F_2(u^R(t))u^R(t))_{L^2_x}\d \beta_t,
\end{eqnarray*}
where $F_1$ is defined by \eqref{eq:defF1} and $F_2$ by \eqref{eq:defF2}. The application of Ito's formula may be justified by an approximation argument as in \cite{debouard2003snls}. Note that for $X \in L^2(\R^d)$, by a simple computation, the following holds
\begin{equation*}
    (X, F_1(X)X)_{L^2_x} = |L X|^2_{L^2_x} - 2 \langle L \rangle_X + \langle L \rangle_X^2|X|^2_{L^2_x},
\end{equation*}
and
\begin{equation*}
    |F_2(X)X|^2_{L^2_x} = |L X|^2_{L^2_x} - 2 \langle L \rangle_X + \langle L \rangle_X^2|X|^2_{L^2_x}.
\end{equation*}
By cancelling these two terms and expanding the last one, we find
\begin{equation}
    \label{eq:sdenormsimpl}
    \d |u^R(t)|^2_{L^2_x} = \Theta_R(u^R\restrict{[0,t]}) \langle L \rangle_{u^R(t)} \left(1-|u^R(t)|^2_{L^2_x}\right) \d \beta_t.
\end{equation}
Define the following process for $t \in [0,T_0]$:
\begin{equation*}
    C_t :=\Theta_R(u^R\restrict{[0,t]}) \langle L \rangle_{u^R(t)}.
\end{equation*}
A direct computation shows that, by \eqref{eq:basicestim} and the definition of $\Theta$,
\begin{equation*}
    |C_t| \leq \| L \| |u^R(t)|^2_{L^2_x}.
\end{equation*}
Therefore, the lemma follows directly from \eqref{eq:sdenormsimpl} and an application of Lemma \ref{lem:SDE}, with $M_t = |u^R(t)|^2_{L^2_x}$.
\end{proof}
Lemma \ref{lem:estimnormL2} shows that, for $R \geq 1$, $\P$-almost surely, for all $t \in [0,T_0]$
\begin{equation*}
    \Theta_R(u^R\restrict{[0,t]}) = 1,
\end{equation*}
and thus $u^R$ is in fact a solution of \eqref{eq:interm}. Similar arguments as in the proof of Lemma \ref{lem:estimnormL2} can be used to show that if $u\in L^2(\Omega, C([0,T_0];L^2_x))$ is a solution of \eqref{eq:interm}, then almost surely, for all $t \in [0,T_0]$
\begin{equation*}
    |u(t)|_{L^2_x} = 1.
\end{equation*}
This implies that $u$ is also a solution of the truncated equation for $R=1$. Therefore, by pathwise uniqueness for the truncated equation, we directly find that $u$ is pathwise unique. This concludes the proof of Proposition \ref{prop:meq-interm}.
\hfill
$\square$

\subsubsection{Proof of Theorem \ref{thm:meq}}

We are now ready to finish the proof of Theorem \ref{thm:meq}. For that, we will use once again a fixed point argument. For $T>0$ and a deterministic function $\xi \in C([0,T], L^1(\R^d, \R))$, we denote $u(\xi)$ the unique strong solution in $L^2(\Omega, C([0,T]; L^2(\R^d)))$ of \eqref{eq:interm} given by Proposition \ref{prop:meq-interm}.

We define the following map, for $\xi \in C([0,T], L^1(\R^d, \R))$ and $t \in [0, T_0]$:
\begin{equation}
\hat{\Tc}(\xi)(t)(\cdot) := \E[|u(\xi)(t, \cdot)|^2].
\end{equation}
Before starting the proof, we state a useful bound. For $f, g \in L^2(\Omega, L^2(\R^d))$ with almost surely unit $L^2_x$-norm, by Cauchy-Schwarz inequality:
\begin{equation}
\label{eq:triang}
    \left|\E |f(\cdot)|^2 - \E |g(\cdot)|^2\right|_{L^1_x} \leq 2 \left|\E |f-g|^2_{L^2_x} \right|^{1/2}.
\end{equation}
Note that the following inclusion holds:
\begin{equation*}
    L^2(\Omega ; C([0,T], L^2(\R^d))) \subset C([0,T], L^2(\Omega ; L^2(\R^d))).
\end{equation*}
From this inclusion and \eqref{eq:triang}, it follows directly that $\hat{\Tc}$ maps $ C([0,T], L^1(\R^d, \R))$ to itself. 

In the following, we show that $\hat{\Tc}$ is a contraction for $T$ small enough.
Let $\xi, \zeta \in C([0,T], L^1(\R^d, \R))$, then by Ito's formula and \eqref{eq:trunc}:
\begin{equation*}
    \begin{aligned}
     &\d |u(\xi)(t) - u(\zeta)(t)|^2_{L^2_x} = 2 (u(\xi)(t) - u(\zeta)(t), - \i V \star (\xi_t - \zeta_t) u(\xi)(t))_{L^2_x}\d t \\
&+ 2 (u(\xi)(t) - u(\zeta)(t), -\i [H + V \star \zeta_t](u(\xi)(t) - u(\zeta)(t)))_{L^2_x} \d t \\
&+ (u(\xi)(t) - u(\zeta)(t), F_1(u(\xi)(t))u(\xi)(t) - F_1(u(\zeta)(t))u(\zeta)(t))_{L^2_x} \d t \\
&+ 2 (u(\xi)(t) - u(\zeta)(t), F_2(u(\xi)(t))u(\xi)(t) - F_2(u(\zeta)(t))u(\zeta)(t))_{L^2_x} \d \beta_t\\
&+|F_2(u(\xi)(t))u(\xi)(t) - F_2(u(\zeta)(t))u(\zeta)(t)|_{L^2_x}^2\d t.
    \end{aligned}
\end{equation*}

Notice that, except for the first term on the right hand side, the required bounds are similar to those in the proof of Proposition \ref{prop:eqtrunc} with $R=1$. Indeed, a similar bound can be derived using estimates \eqref{eq:estimatef12} and \eqref{eq:estimatef22} from Lemma \ref{lem:estimateF} and the preservation of the $L^2_x$-norm. Thus, there exists an explicit constant $C_1 > 0$, depending only of $|V|_{L^{\infty}_x}, \| L \|$ and $T_0$ such that
\begin{equation*}
    \begin{aligned}
     \E [|u(\xi)(t) - u(\zeta)(t)|^2_{L^2_x}] &\leq \E \int_0^t (u(\xi)(s) - u(\zeta)(s), - \i V \star (\xi_s - \zeta_s) u(\xi)(s))_{L^2_x}\d s \\
     &+ C_1 \int_0^t \E [|u(\xi)(s) - u(\zeta)(s)|^2_{L^2_x}] \d s.
    \end{aligned}
\end{equation*}
For the first term on the right hand side, notice that
\begin{equation*}
    \begin{aligned}
     &|(u(\xi)(s) - u(\zeta)(s), - \i V \star (\xi_s - \zeta_s) u(\xi)(s))_{L^2_x}|\\
     &\leq |V \star (\xi_s - \zeta_s)|_{L^{\infty}_x} |u(\xi)(s) - u(\zeta)(s)|_{L^2_x} |u(\xi)(s)|_{L^2_x} \\
&\leq \Frac{1}{2} |V \star (\xi_s - \zeta_s)|_{L^{\infty}_x}^2 + \Frac{1}{2}|u(\xi)(s) - u(\zeta)(s)|_{L^2_x}^2\\
&\leq \Frac{1}{2}|V|_{L^{\infty}_x}^2|\xi_s - \zeta_s|_{L^1_x}^2 + \Frac{1}{2}|u(\xi)(s) - u(\zeta)(s)|_{L^2_x}^2,
    \end{aligned}
\end{equation*}
by the almost sure preservation of the $L^2_x$-norm for all time. Thus, collecting the results gives, for all $t \leq T$
\begin{equation}
    \label{eq:gronwalltronc}
\begin{aligned}
 \E [|u(\xi)(t) - u(\zeta)(t)|^2_{L^2_x}] &\leq \Frac{1}{2}|V|_{L^{\infty}_x}^2 T |\xi - \zeta|_{L^{\infty}(0, T; L^1_x)}^2 \\
 &+ (C_1 + \Frac{1}{2}) \int_0^t \E [|u(\xi)(s) - u(\zeta)(s)|^2_{L^2_x}] \d s.
\end{aligned}
\end{equation}
By Gronwall's lemma,  \eqref{eq:triang} and taking the supremum over $[0,T]$ we obtain
\begin{equation*}
    |\hat{\Tc}\xi - \hat{\Tc}\zeta|_{L^{\infty}(0,T;L^1_x)} \leq \sqrt{2}|V|_{L^{\infty}_x} e^{(C_1 + 1/2) T_0/2} \sqrt{T} |\xi - \zeta|_{L^{\infty}(0, T; L^1_x)}.
\end{equation*}
Therefore, by taking $T\in (0, T_0]$ such that
\begin{equation*}
    T \leq \inf \left(T_0, \frac{e^{-(C_1 + 1/2) T_0}}{2|V|_{L^{\infty}_x}^2} \right),
\end{equation*}
the application $\hat{\Tc}$ maps $C([0,T], L^1(\R^d; \R))$ into itself and is a contraction. Hence, $\Tc$ has a unique fixed point in $C([0,T], L^1(\R^d; \R))$, which is the solution to \eqref{eq:mfeq}. Moreover, the choice of $T$ is independent of $u_0$ since $C_1$ is itself independent of $u_0$. Then, a similar decomposition of $[0,T_0]$ as in the proof of Proposition \ref{prop:eqtrunc} can be done. The result can therefore easily be extended to the time interval $[0,T_0]$, which concludes the proof of Theorem \ref{thm:meq}.
\hfill
$\square$

\section{Propagation of the regularity}\label{sec:propreg}

In order to prove the convergence of the $N$-body system toward the mean-field equation, a bound on the $H^1_x$-norm of the solutions of the mean-field equation \eqref{eq:mfeq} is needed. Thus, we begin by showing that, under more restrictive assumptions, Theorem \ref{thm:meqh1} holds i.e. we have $H^1$ solutions. Then we will show a moment estimate for these solutions.

\subsection{Proof of Theorem \ref{thm:meqh1}}

To prove Theorem \ref{thm:meqh1}, we follow the same steps as in Section \ref{sec:mfeq}. We start by proving the result for the intermediate equation \eqref{eq:interm}, then for the mean-field one. More precisely, we prove the following proposition:

\begin{proposition}\label{prop:eqtrunch1}

Let $T_0 >0$. Assume that $H = - \Delta$, the potential $V$ is in $L^{\infty}(\R^3)$ and even and that $u_0 \in L^2(\Omega, H^1(\R^3))$ is a $\Fc_0$-measurable random variable such that $|u_0|_{L^2_x} = 1$ almost surely. In addition, assume that the coupling operator $L\in \Lc(L^2_x)$ satisfies the following conditions $[\nabla, L ] \in \Lc( H^1_x, L^2_x)$ and $[\nabla, L^* L ] \in \Lc( H^1_x, L^2_x)$. Finally, assume that $\xi$ is in $ C([0,T_0],W^{1,1}(\R^d; \R))$ such that $|\xi_t|_{L^1_x}=1$ for all $t \leq T_0$ . Then the unique solution $u$ of \eqref{eq:trunc} such that $u(0) = u_0$ given by Proposition \ref{prop:meq-interm} lies in $L^2(\Omega, C([0,T_0]; H^1(\R^3)))$. 
\end{proposition}

Before starting the proof, we state new key estimates using the additional assumptions on $L$, where the definition of $F_1$ is given by \eqref{eq:defF1} and $F_2$ by \eqref{eq:defF2}.

\begin{lemma}\label{lem:estimFh1}
Fix $R>0$. Let $X \in H^1(\R^d)$, such that $|X|_{L^2_x}  \leq 2 R$. Assume that $L$ is a bounded operator such that $[\nabla, L ] \in \Lc( H^1_x, L^2_x)$ and $[\nabla, L^* L ] \in \Lc( H^1_x, L^2_x)$. Then the following estimates holds:
\begin{itemize}
    \item There exists $C_1(R, L) > 0$ such that
    \begin{equation}\label{eq:estimf1h1}
        |\nabla F_1(X) X|_{L^2_x} \leq C_1(R, L) |X|_{H^1_x}.
    \end{equation}
    \item There exists $C_2(R, L) > 0$ such that
    \begin{equation}\label{eq:estimf2h1}
        |\nabla F_2(X) X|_{L^2_x} \leq C_2(R, L) |X|_{H^1_x}.
    \end{equation}
\end{itemize}
\end{lemma}

The proof of Lemma \ref{lem:estimFh1} can be found in Appendix \ref{sec:appendix-estF}. Let $R \geq 1$ that will be fixed later to $1$, recall the definition of the application $\Tc_R$ from the proof of Proposition \ref{prop:eqtrunc}, for $u \in C_{T_0} L^2$
\begin{equation}
    \label{eq:rappelmild}
    \begin{aligned}
    \Tc_R u(t) &= S(t) u_0 - \i \int_0^t S(t-s) V \star \xi_s u(s) \d s  \\
&- \frac{1}{2} \int_0^t  \Theta_R(u\restrict{[0,s]})^2 S(t-s) [F_1(u(s))u(s)]\d s \\
&+ \int_0^t\Theta_R(u\restrict{[0,s]}) S(t-s) [F_2(u(s))u(s)] \d \beta_s.
    \end{aligned}
\end{equation}
We now give a key estimate on the application $\Tc_R$ which will be used twice to show the preservation of the regularity.
\begin{lemma}\label{lem:boundreg}
Let $T_0>0$, $R \geq 1$. Assume that $u_0$ and $\xi$ satisfy the same assumptions as in Proposition \ref{prop:eqtrunch1}. Then, there exist $C_1, C_2 >0$ depending only on $|V|_{L^{\infty}_x}, L, R$ and $T_0$ such that for all $u \in L^2(\Omega; C_{T_0}H^1)$, for all $t \leq T_0$
\begin{equation}
    \label{eq:boundgronwreg}
    \begin{aligned}
     \E[\sup_{s \leq t}|\Tc_R u(s)|_{H^1_x}^2] &\leq C_1 \big(\E[|u_0|^2_{H^1_x}] + \int_0^t |\xi_s|_{W^{1,1}_x}^2 \d s \int_0^t \E[|u(s)|_{L^2_x}^2] \d s \big) \\
     &+ C_2 \int_0^t \E[|u(s)|^2_{H^1_x}] \d s.
    \end{aligned}
\end{equation}
\end{lemma}

\begin{proof}[Proof of Lemma \ref{lem:boundreg}]
Note that, if $H = - \Delta$, then the unitary group $(S(t))_{t \in \R}$ and $\nabla$ commute, which allows us to take the gradient of \eqref{eq:rappelmild}. Hence, since $(S(t))_{t \in \R}$ is a unitary group we get
\begin{equation*}
    \begin{aligned}
    |\nabla \Tc_R u(t)|_{L^2_x} - |\nabla u_0|_{L^2_x} &\leq \int_0^t |\nabla [(V \star \xi_s) u(s)]|_{L^2_x} \d s  + \frac{1}{2} \int_0^t  \Theta_R(u\restrict{[0,s]})^2 |\nabla F_1(u(s))u(s)|_{L^2_x}\d s \\
&+ |\int_0^t\Theta_R(u\restrict{[0,s]}) S(t-s) \nabla F_2(u(s))u(s) \d \beta_s|_{L^2_x} \\
&=  I + \Frac{1}{2}II + III. 
    \end{aligned}
\end{equation*}
For $I$, observe that
\begin{equation*}
    \begin{aligned}
    I &\leq  \int_0^t |\nabla (V \star \xi_s) u(s)|_{L^2_x} \d s +  \int_0^t |V \star \xi_s \nabla u(s)|_{L^2_x} \d s \\
    &\leq \int_0^t |V \star \nabla \xi_s|_{L^{\infty}_x} |u(s)|_{L^2_x} \d s + \int_0^t |V \star \xi_s|_{L^{\infty}_x} |u(s)|_{H^1_x} \d s \\
    &\leq |V|_{L^{\infty}_x} \int_0^t |\nabla \xi_s|_{L^1_x}|u(s)|_{L^2_x} +  |u(s)|_{H^1_x} \d s,
    \end{aligned}
\end{equation*}
where we used a basic convolution inequality. Hence we get by Cauchy-Schwarz inequality
\begin{equation}
    \label{eq:solh1est1}
    I \leq |V|_{L^{\infty}_x} \big(\int_0^t |\nabla \xi_s|_{L^1_x}^2\d s \int_0^t |u(s)|_{L^2_x}^2 \d s\big)^{1/2} + |V|_{L^{\infty}_x} \sqrt{T_0} \big(\int_0^t |u(s)|_{H^1_x}^2 \d s\big)^{1/2}.
\end{equation}
For $II$, we directly use \eqref{eq:estimf1h1} from Lemma \ref{lem:estimFh1} and Cauchy-Schwarz inequality:
\begin{equation}
    \label{eq:solh1est2}
    \begin{aligned}
    II &\leq C_1 \int_0^t |u(s)|_{H^1_x} \d s\\
    &\leq C_1 \sqrt{T_0} \big(\int_0^t |u(s)|_{H^1_x}^2 \d s\big)^{1/2}.
    \end{aligned}
\end{equation}
For $III$, we start by using Lemma \ref{lem:itoiso} with
\begin{equation*}
    \Phi_s := \Theta_R(u\restrict{[0,s]}) \nabla F_2(u(s))u(s).
\end{equation*}
Then, by \eqref{eq:estimf2h1} from Lemma \ref{lem:estimFh1} we find
\begin{equation}
    \label{eq:solh1est3}
    \E \left[\sup_{t \leq T} |\int_0^t\Theta_R(u\restrict{[0,s]}) S(t-s) \nabla F_2(u(s))u(s) \d \beta_s|_{L^2_x}^2\right] \leq C_2 \int_0^t \E |u(s)|_{H^1_x}^2 \d s.
\end{equation}
Therefore, combining \eqref{eq:solh1est1}-\eqref{eq:solh1est3} with Young's inequality we obtain that there exist constant $C_3, C_4 >0$ depending on $R, L, |V|_{L^{\infty}_x}$ and $T_0$ such that
\begin{equation}
    \label{eq:solh1est4}
    \begin{aligned}
     \E[\sup_{s \leq t}|\nabla \Tc_R u(s)|_{L^2_x}^2] &\leq C_3 \big(\E|\nabla u_0|^2_{L^2_x}] + 1 + \int_0^t |\nabla \xi_s|_{L^1_x}^2 \d s \int_0^t \E[|u(s)|_{L^2_x}^2] \d s \big) \\
     &+ C_4 \int_0^t \E[|u(s)|^2_{H^1_x}] \d s.
    \end{aligned}
\end{equation}
In a similar and simpler way, we can show that, starting directly from \eqref{eq:rappelmild} and  using Lemma \ref{lem:estimateF}, and $|u_0|_{L^2_x} = 1$ a.s., there exists $C_5>0$ (depending on $R, L, |V|_{L^{\infty}_x}$ and $T_0$) such that
\begin{equation}
\label{eq:solh1est5}
    \E[\sup_{s \leq t}| \Tc_R u(s)|_{L^2_x}^2] \leq C_5 \big(1 + \int_0^t | \xi_s|_{L^1_x}^2 \d s\int_0^t \E[|u(s)|_{L^2_x}^2] \d s \big).
\end{equation}
Hence, collecting \eqref{eq:solh1est4} and \eqref{eq:solh1est5}, the desired inequality holds.
\end{proof}
We are now ready to tackle the proofs of Proposition \ref{prop:eqtrunch1} and Theorem \ref{thm:meqh1}.

\begin{proof}[Proof of Proposition \ref{prop:eqtrunch1}]
Fix $R = 1$. By Lemma \ref{lem:boundreg}, there exist $C_1, C_2 >0$ depending on $|V|_{L^{\infty}_x}, L$ and $T_0$ such that for all $t \leq T_0$
\begin{equation*}
    \begin{aligned}
     \E[\sup_{s \leq t}|\Tc_1 u(s)|_{H^1_x}^2] &\leq C_1 \big(\E|u_0|^2_{H^1_x} + \int_0^t |\xi_s|_{W^{1,1}_x}^2 \d s \int_0^t \E[|u(s)|_{L^2_x}^2] \d s \big) \\
     &+ C_2 \int_0^t \E[|u(s)|^2_{H^1_x}] \d s.
    \end{aligned}
\end{equation*}
Since $\xi$ is fixed, by taking the supremum in time for $T \leq T_0$, we find
\begin{equation}
    \label{eq:presreginter}
    \E[|\Tc_1 u|_{L^{\infty}(0,T;H^1_x)}^2] \leq C_1 \E|u_0|^2_{H^1_x} + T \big(C_1 T_0 |\xi|_{L^{\infty}(0,T_0;W^{1,1}_x)}^2 + C_2\big) \E|u|_{L^{\infty}(0,T;H^1_x)}^2.
\end{equation}
By the proofs of Proposition \ref{prop:meq-interm} and of Proposition \ref{prop:eqtrunc} and Banach fixed point theorem, for $T$ small enough, there exists an approximation sequence $(u_n)_{n \in \N} \subset L^2(\Omega, C_TL^2)$ such that, for all $n \geq 1$
\begin{equation*}
    u_n = \Tc_1(u_{n-1}),
\end{equation*}
which converges in $L^2(\Omega, C_TL^2)$ to $u(\xi)$, the solution of \eqref{eq:interm} such that $u(\xi)(0) = u_0$ given by Proposition \ref{prop:eqtrunc}. To show that $u(\xi) \in L^2(\Omega, L^{\infty}(0,T;H^1_x))$, by classical arguments of weak-* convergence it is enough to show that the sequence $u_n$ is uniformly bounded in this space. This follows directly from \eqref{eq:presreginter}. Then, similarly as in the proof of Proposition \ref{prop:eqtrunc}, the result can be extended to the time interval $[0,T_0]$ by pathwise uniqueness. Finally, the continuity in $H^1(\R^d)$ follows from the mild form of the equation and the time continuity of the group $S$. Therefore $u(\xi)$ lies in $L^2(\Omega, C_{T_0}H^1)$, as required.
\end{proof}

\begin{proof}[Proof of Theorem \ref{thm:meqh1}]
First, we recall that for $\xi \in C([0,T_0]; L^1_x)$, we denote by $u(\xi)$ the unique solution of \eqref{eq:interm} in $L^2(\Omega, C_{T_0}L^2)$ such that $u(\xi)(0) = u_0$ given by Proposition \ref{prop:meq-interm}. We also recall that the mapping on $C([0,T_0]; L^1_x)$ defined by:
for $\xi \in C([0,T], L^1(\R^d, \R))$ and $t \in [0, T_0]$:
\begin{equation}
\hat{\Tc}(\xi)(t)(\cdot) := \E[|u(\xi)(t, \cdot)|^2],
\end{equation}
is a contraction for a time $T_1$ small enough, and $T_1$ depend only on $|V|_{L^{\infty}_x}, L$ and $T_0$. We set $\xi_{0,t}(\cdot) = \E |u_0(\cdot)|^2$ for all $t\leq T_0$. Since $\hat{\Tc}$ is a contraction mapping on $C([0,T_1]; L^1_x)$ the sequence $(\xi_n)_{n \in \N}$ defined, for $n \geq 0$, by
\begin{equation}
    \xi_{n+1} = \hat{\Tc} \xi_n,
\end{equation}
is converging to the fixed point $\xi_{\infty}$ in $C([0,T_0]; L^1_x)$. Now, we define a new sequence $(u_n)_{n \in \N} \subset L^2(\Omega, C_{T_1}L^2)$ by
\begin{equation}
    u_n := u(\xi_n),
\end{equation}
and note that, from this definition, $\xi_{n+1}(t)(\cdot) = \E[|u_n(t, \cdot)|^2]$ for $t \leq T_1$. We aim to show that $u(\xi_{\infty})$ lies in $L^2(\Omega, C_{T_1}H^1)$, which is the unique solution of \eqref{eq:mfeq} such that $u(0) = u_0$. To achieve this, we first prove that $(u_n)_{n \in \N}$ converges to $u(\xi_{\infty})$ in $L^2(\Omega, C_{T_1}L^2)$ and second we prove a similar bound as \eqref{eq:presreginter}. This will enable us to show the $H^1$ regularity of $u(\xi_{\infty})$ in a similar way as in the end of the proof of Proposition \ref{prop:eqtrunch1}.

First, we show the convergence of the sequence $(u_n)_{n \in \N}$ to the desired limit. We start by observing that we can improve the bound \eqref{eq:gronwalltronc} in the proof of Theorem \ref{thm:meq} by taking the supremum in time directly in Ito's formula, before taking the expectation. Then we can apply Lemma \ref{lem:itoiso} and Lemma \ref{lem:estimateF} as we did several times earlier in this article. This gives that there exist $C_3, C_4 >0$ (depending only on $|V|_{L^{\infty}_x}, L, T_0$ and $T_1$) such that for all $t \leq T_1$
\begin{equation*}
    \begin{aligned}
 \E [|u(\xi) - u(\zeta)|^2_{C_t L^2}] &\leq T_1 C_3 |\xi - \zeta|_{L^{\infty}(0, T_1; L^1_x)}^2 \\
 &+ C_4 \int_0^t \E [|u(\xi) - u(\zeta)|^2_{C_s L^2}] \d s,
\end{aligned}
\end{equation*}
where $\xi, \zeta \in C([0,T_1]; L^1_x)$. Thus, by applying Gronwall's lemma, it holds that there exists a constant $C_5>0$ depending only on $|V|_{L^{\infty}_x}, L, T_0$ and $T_1$ such that for all $\xi, \zeta \in C([0,T_1]; L^1_x)$
\begin{equation*}
    |u(\xi) - u(\zeta)|_{L^2(\Omega, C_{T_1}L^2)} \leq C_5 |\xi - \zeta|_{L^{\infty}(0, T_1; L^1_x)}.
\end{equation*}
The convergence of the sequence $(u_n)_{n\in\N}$ to $u(\xi_{\infty})$ in $L^2(\Omega, C_{T_1}L^2)$ directly follows from this inequality and the convergence of the sequence $(\xi_n)_{n\in \N}$  to $\xi_{\infty}$ in $C([0,T_1]; L^1_x)$.

Second, we show a similar bound as \eqref{eq:presreginter}. Note that since $|u_0|_{L^2_x} = 1$ almost surely, by the preservation of the $L^2_x$-norm of $u(\xi)$ given by Proposition \ref{prop:meq-interm}, we have for all $n \in \N$ and all $t \leq T_1$
\begin{equation*}
    |\xi_{n,t}|_{L^1_x} = 1.
\end{equation*}
Moreover, note that by definition of the sequence $(u_n)_{n\in \N}$
\begin{equation*}
    \nabla \xi_{n,t}(x) = 2 \E \big[ \Re \: \overline{u_{n-1}(t,x)} \nabla u_{n-1}(t,x) \big],
\end{equation*}
for all $x \in \R^3$ and all $t \leq T_1$. Hence, using twice Cauchy-Schwarz inequality and the preservation of the $L^2_x$-norm of $u_{n-1}$, we obtain
\begin{equation}
    \label{eq:boundw11}
    |\nabla \xi_{n,t}|_{L^1_x} \leq 2 \E[|u_{n-1}(t)|_{L^2_x} |\nabla u_{n-1}(t)|_{L^2_x}] \leq 2 \big(\E[|\nabla u_{n-1}(t)|^2_{L^2_x}]\big)^{1/2}.
\end{equation}
Thus, if $u_{n-1} \in L^2(\Omega, C_{T_1}H^1)$, then $\xi_n \in C([0,T_1];W^{1,1}_x)$. Therefore since $\xi_0$ is in $C([0,T_1];W^{1,1}_x)$, by induction and applying Proposition \ref{prop:eqtrunch1} at each step, $u_n$ lies in the space $L^2(\Omega, C_{T_1}H^1)$ for all $n \in \N$. For now, let $n \geq 1$. Then, since $u_n$ is the fixed point of $\Tc_1$ for $\xi_n$ (see the proof of Proposition \ref{prop:eqtrunch1}), \eqref{eq:boundgronwreg} implies there exist $C_1, C_2 >0$ such that, for all $t \leq T_1$
\begin{equation*}
    \begin{aligned}
     \E[\sup_{s \leq t}|u_n(s)|_{H^1_x}^2] &\leq C_1 \big(\E[|u_0|^2_{H^1_x}] + \int_0^t |\xi_{n,s}|_{W^{1,1}_x}^2 \d s \int_0^t \E[|u_n(s)|_{L^2_x}^2] \d s \big) \\
     &+ C_2 \int_0^t \E[|u_n(s)|^2_{H^1_x}] \d s.
     \end{aligned}
\end{equation*}
Using the preservation of the $L^2_x$-norm of $u_n$ and Gronwall's lemma, we find
\begin{equation}
    \label{eq:breginter}
    \begin{aligned}
     \E |u_n|_{C_{T_1}H^1}^2 &\leq C_1 e^{C_2 T_1} \big(\E[|u_0|^2_{H^1_x}] + \int_0^{T_1} |\xi_{n,s}|_{W^{1,1}_x}^2 \d s \big) \\
     &\leq C_1 e^{C_2 T_1} \big(\E[|u_0|^2_{H^1_x}] + T_1 |\xi_{n}|_{L^{\infty}(0,T;W^{1,1}_x)}^2\big).
    \end{aligned}
\end{equation}
Combining \eqref{eq:breginter} and \eqref{eq:boundw11} yields the existence of $\hat{C}_1, \hat{C}_2 >0$ depending only of $L, T_0, T_1$ and $|V|_{L^{\infty}_x}$ such that 
\begin{equation}
    \E |u_n|_{C_{T_1}H^1}^2 \leq \hat{C}_1 \E[|u_0|^2_{H^1_x}] + T_1 \hat{C}_2 \E |u_{n-1}|_{C_{T}H^1}^2.
\end{equation}
Therefore, using the convergence of the sequence $(u_n)_{n \in \N}$ to $u(\xi_{\infty})$ in $L^2(\Omega, C_{T_1}L^2)$ and the above inequality we can conclude similarly as in the proof of Proposition \ref{prop:eqtrunch1}. Hence, $u(\xi_{\infty}) \in L^2(\Omega, C_{T_1}H^1)$ and the time interval can be extended to $[0,T_0]$ by pathwise uniqueness in a similar way as in the proof of Proposition \ref{prop:eqtrunc}. Finally, since $\xi_{\infty}$ is the fixed point of $\hat{\Tc}$, $u(\xi_{\infty})$ is the unique solution of \eqref{eq:mfeq} such that $u(0) = u_0$ given by Theorem \ref{thm:meq}, achieving the proof of Theorem \ref{thm:meqh1}.
\end{proof}

\subsection{Moment estimate}

We are now equipped to obtain a bound on the moments of the $H^1$-norm of the solution of \eqref{eq:mfeq}.

\begin{proposition}
\label{prop:borneh1}
Let $p\geq 4$ and assume $H, V, L$ and $u_0$ satisfy the assumptions of Proposition \ref{prop:eqtrunch1}. Moreover, assume that $u_0\in L^p(\Omega, H^1(\R^d))$. Let $u$ be the solution of \eqref{eq:mfeq} with initial condition $u_0$. Then there exists a constant $C>0$ depending only on $p, T_0, \| L \|, \| [\nabla, L ]\|_{\Lc( L^2_x, H^1_x)}$ and $|V|_{L^{\infty}_x}$ such that
\begin{equation*}
    \sup_{t \leq T_0} \E [|u(t)|^p_{H^1_x}] \leq C  \E [|u_0|^p_{H^1_x}].
\end{equation*}
\end{proposition}

\begin{proof}Let $u$ be the solution of \eqref{eq:mfeq}. To begin, notice that, by taking the gradient, it holds
\begin{equation*}
    \begin{aligned}
     \d \nabla u(t) &= - \i ( \Delta \nabla + \nabla(W_t) + W_t \nabla) u(t) \d t - \Frac{1}{2}\nabla F_1(u(t)) u(t) \d t \\
     &+ \nabla F_2(u(t)) u(t) \d \beta_t,
    \end{aligned}
\end{equation*}
where $W_t = V \star \E|u(t)(\cdot)|^2$. Thus the following expression for the process $(|\nabla u(t)|_{L^2_x}^2)_{t\leq T_0}$ holds by Ito's formula:
\begin{equation*}
    \begin{aligned}
        \d |\nabla u(t)|^2 
&= 2 (\nabla u(t), -\i (\nabla W_t) u(t))_{L^2_x} \d t + ( \nabla u(t), - \nabla F_1(u(t)) u(t))_{L^2_x} \d t \\
&+ |\nabla F_2(u(t)) u(t)|^2_{L^2_x} \d t + 2 (\nabla u(t), \nabla F_2(u(t)) u(t))_{L^2_x} \d \beta_t,
    \end{aligned}
\end{equation*}
by self-adjointness of the multiplication by $W_t$ and of $\Delta$ similarly as in the proof of Lemma \ref{lem:estimnormL2}.

Now, we can bound the moments with a localization arguments. For $K > 0$, set the following stopping time:
\begin{equation*}
    \tau_K = \inf \{ t >0, |\nabla u(t)|^2_{L^2_x} \geq K\} \wedge T_0
\end{equation*}
In the following we will denote $u^K$ the process $u$ stopped by the stopping time $\tau_K$. Then, for $q \geq 2$ such that $p = 2q$, applying Ito's lemma yields
\begin{equation*}
\begin{aligned}
\d (|\nabla u^K(t)|_{L^2_x}^2)^q  &= 2 q |\nabla u^K(t)|_{L^2_x}^{2(q-1)}(\nabla u^K(t), -\i (\nabla W_t) u^K(t))_{L^2_x} \d t \\
&+ q |\nabla u^K(t)|_{L^2_x}^{2(q-1)}(\nabla u^K(t), - \nabla F_1(u^K(t)) u^K(t) )_{L^2_x} \d t \\
&+ q |\nabla u^K(t)|_{L^2_x}^{2(q-1)}|\nabla F_2(u^K(t)) u^K(t)|_{L^2_x}^2\d t \\
&+ 2 q |\nabla u^K(t)|_{L^2_x}^{2(q-1)} (\nabla u^K(t), \nabla F_2(u^K(t))u^K(t) )_{L^2_x}\d \beta_t \\
&+ 2 q (q-1) |\nabla u^K(t)|_{L^2_x}^{2(q-2)}|(\nabla u^K(t), \nabla F_2(u^K(t))u^K(t) )_{L^2_x}|^2 \d t.
\end{aligned}
\end{equation*}
The stopping time enables us to take the expectation, since the stopped martingale part is square integrable thanks to Lemma \ref{lem:estimFh1}. Applying Cauchy-Schwarz inequality on each scalar product gives, after integrating in time
\begin{equation*}
    \begin{aligned}
    \E |\nabla u^K(t)|_{L^2_x}^{2q} - \E |\nabla u^K(0)|_{L^2_x}^{2q} &\leq 2 q \int_0^t \E [|\nabla u^K(s)|_{L^2_x}^{2q -1} |\nabla W_s u^K(s)|_{L^2_x}] \d s \\
    &+ q \int_0^t \E [|\nabla u^K(s)|_{L^2_x}^{2q - 1} |\nabla F_1(u^K(s))u^K(s)|_{L^2_x}] \d s \\
    &+ q \int_0^t \E [|\nabla u^K(s)|_{L^2_x}^{2q -2}|\nabla F_2(u^K(s)) u^K(s)|^2_{L^2_x}] \d s \\
    &+ 2q(q-1) \int_0^t \E [|\nabla u^K(s)|_{L^2_x}^{2q - 2} |\nabla F_2 u^K(s) u^K(s)|^2_{L^2_x}] \d s\\
    &= I + II + III + IV.
    \end{aligned}
\end{equation*}
For $I$, since $u^K$ has almost surely unit $L^2_x$-norm, we have by Cauchy-Schwarz inequality
\begin{equation*}
    \begin{aligned}
    |\nabla W_s u^K(s)|_{L^2_x} &\leq |V \star \nabla \E[|u^K(s)|^2] |_{L^{\infty}_x} |u^K(s)|_{L^2_x} \\
    &= |V \star (\E[ \overline{u^K(s)} \nabla u^K(s) + u^K(s) \overline{\nabla u^K(s)}])|_{L^{\infty}_x} \\
    &\leq 2 |V|_{L^{\infty}_x} \E |\nabla u^K(s)|_{L^2_x}.
    \end{aligned}
\end{equation*}
Therefore, by applying Holder's inequality on both expectations
\begin{equation}
    \label{eq:bdh1estimI}
    \begin{aligned}
    I &\leq 4q |V|_{L^{\infty}_x} \int_0^t \E[|\nabla u^K(s)|_{L^2_x}^{2q -1}] \E[|\nabla u^K(s)|_{L^2_x}] \d s \\
    &\leq 4q |V|_{L^{\infty}_x} \int_0^t \E[|\nabla u^K(s)|_{L^2_x}^{2q}]\d s.
    \end{aligned}
\end{equation}

For $II$, by applying Lemma \ref{lem:estimFh1}:
\begin{equation*}
    \begin{aligned}
    II &\leq q C_1 \Int_0^t \E [|\nabla u^K(s)|_{L^2_x}^{2q - 1} |u^K(s)|_{H^1_x}]\\
    &\leq q C_1 \Int_0^t \E [|u^K(s)|_{H^1_x}^{2q}] \d s.
    \end{aligned}
\end{equation*}
The terms $III$ and $IV$ are treated in a similar way as $II$. Hence there exists $C_2 >0$ independent of $K$ such that
\begin{equation}
    \label{eq:bdh1estimII}
    II + III + IV \leq C_2 \int_0^t \E [|u^K(s)|^{2q}_{H^1_x}] \d s.
\end{equation}
Combining \eqref{eq:bdh1estimI} and \eqref{eq:bdh1estimII} and the fact that $u^K$ has unit $L^2_x$-norm, there exists $C_3 >0$ independent of $K$ such that for all $t \leq T_0$
\begin{equation*}
    \E [|u^K(t)|^{2q}_{H^1_x}]\leq 2^{2q} \E [|u_0|^{2q}_{H^1_x}] + C_3 \int_0^t \E [|u^K(s)|^{2q}_{H^1_x}] \d s.
\end{equation*}
Then, by Gronwall's lemma
\begin{equation}
    \sup_{t \in [0,T_0]} \E [|u^K(t)|^{2q}_{H^1_x}] \leq 2^{2q} e^{C_3 T_0} \E [|u_0|^{2q}_{H^1_x}].
\end{equation}
Finally, as the trajectories are continuous in $H^1(\R^3)$, we can conclude the localization argument.
\end{proof}

\section{Mean-field limit}\label{sec:mflim}

In this section, we prove Theorem \ref{thm:cvgmf}, which is the rigorous derivation of \eqref{eq:mfeq} from a many body system. We begin by defining a number of objects and notations before moving on to the proof of the theorem itself.

\subsection{Definitions and notations}

In this subsection, we introduce definitions and notations that we will use throughout the proof of Theorem \ref{thm:cvgmf}. We start by giving a formulation as a density operator for the solution of \eqref{eq:mfeq}. Recall that we denote by $\Psi_N$ the solution of \eqref{eq:Npeq} driven by $(B^j)_{j \leq N}$ with initial condition $\phi_0^{\otimes N}$ and $\rho^N$ is its associated density operator (given by \eqref{eq:defrhon}). Since $|\psi_N(t)|_{L^2_x} = 1$ for all $t \leq T_0$ almost surely, by \eqref{eq:defrhon} it easily follows that almost surely for all $t \leq T_0$
\begin{equation*}
    \| \rho_t^N \|_1 = |\psi_N(t)|_{L^2_x}^2 =  1.
\end{equation*}
We also recall that we denote by $\phi^{MF,j}$ a solution to \eqref{eq:mfeq} driven by the Brownian motion $B^j$ with initial condition $\phi_0$. When the Brownian motion chosen is not important, the index $j$ will be omitted. For such solution, recall also the definition of the associated projector:
\begin{equation}
    \label{eq:defpj}
    p_{j,t} := |\phi^{MF,j}_t \rangle \langle \phi^{MF,j}_t| = ( \phi^{MF,j}_t, \cdot)_{L^2_x}\phi^{MF,j}_t.
\end{equation}
Notice that this is in fact the density operator associated to the wave function $\phi^{MF,j}$. Therefore, the dynamics for the operator $p_j$ is given by the mean-field Belavkin equation
\begin{equation}
    \label{eq:BelavkinMF}
    \begin{aligned}
     \mathrm{d}p_{j,t}&= -\i[\Delta  + V^{\xi_t}, p_{j,t}]\mathrm{d}t + \Big(L p_{j,t}L^{*} - \frac{1}{2}\{L^{*}L, p_{j,t}\}\Big)\mathrm{d}t \\
     &+ \Big(L p_{j,t} + p_{j,t}L^{*} - tr\big((L+L^{*})p_{j,t}\big)p_{j,t}\Big)\mathrm{d}B^j_t,
    \end{aligned}
\end{equation}
where $\xi(t) = \E[p_{j,t}(x,x)]$ and $V^{\xi(t)}$ stands for the multiplication operator by the function $V \star \xi(t)$. Here, in the definition of $\xi$, $p_{j,t}(\cdot, \cdot)$ stands for the kernel of the operator $p_{j,t}$. Then notice that, by definition of $\xi$ and $p$ we have, for all $t \leq T_0$ and $x \in \R^3$
\begin{equation*}
    \xi_t(x) = \E[|\phi^{MF,j}_t(x)|^2],
\end{equation*}
which implies that the $L^1_x$-norm of $\xi$ is one.

We also define the orthogonal projector $q_{j,t}$ by
\begin{equation}
\label{eq:defqj}
    q_{j,t} := I_d - p_{j,t}.
\end{equation}
Since $p$ is a rank-one projector, it follows that, almost surely, for all $t \leq T_0$
\begin{equation*}
    \| q_{j,t} \| \leq 1.
\end{equation*}
When the projector is viewed as an operator acting on the tensorized space $L^2(\R^{3N}) = L^2(\R^3)^{\otimes N}$, we denote it by
\begin{equation*}
    p^N_{j,t}:=\underbrace{I_d\otimes\cdots\otimes I_d}_{\mbox{\small $j-1$}} \otimes p_{j,t}\otimes \underbrace{I_d\otimes\cdots\otimes I_d}_{\mbox{\small $N-j$}}.
\end{equation*}

To study the convergence of the many body dynamics towards the mean-field dynamics, we study the indicator defined by \cite{pickl2011simple, knowles2010mean} and developed in \cite{kolokoltsov2022quantum} for such stochastic systems. For $J \subset \{0,\dots, N\}$, this indicator of convergence is defined, for $t \leq T_0$, by:

\begin{equation}
    I^{N,J}(t):= 1-\E (\bigotimes_{j \in J}\phi_{t}^{MF,j}, \rho_t^{N,J} \bigotimes_{j \in J}\phi_{t}^{MF,j})_{L^2(\R^{3|J|})} = 1 - \E[\trace (\rho_t^{N,J} 
 \bigotimes_{j \in J}p_{j,t})].
\end{equation}
Two results make the study of this indicator interesting. The proof of both results can be found in \cite{knowles2010mean}. The first result links $I^{N,J}$ and the trace norm distance of the marginals $R^{N,J}$, where we recall that $R^{N,J}$ is defined by
\begin{equation*}
    R_{N,J}(t) := \E \|\rho^{N,J}_t - \bigotimes_{j \in J} p_{j,t}\|_{\Lc^1(L^2(\R^{d |J|}))}.
\end{equation*}
These two indicators are linked by the following inequalities (see Lemma 2.3 from \cite{knowles2010mean}):
\begin{equation*}
    I^{N,J}(t) \leq R^{N,J}(t) \leq 2\sqrt{2 I^{N,J}(t)}.
\end{equation*}
This relationship shows that the convergence to $0$ of $R^{N,J}(t)$ is equivalent to the convergence of $I^{N,J}(t)$. The second important result on $I^{N,J}(t)$ is the following inequality
\begin{equation*}
    I^{N,J}(t) \leq |J| I^{N,1}(t).
\end{equation*}
This inequality shows that, in order to prove Theorem \ref{thm:cvgmf}, it is sufficient to study the convergence of $I^{N,1}$ to $0$ when $N$ goes to infinity.

We define the following process $\hat{I}^{N,j}_t$ for $j \leq N$ by
\begin{equation}
    \label{eq:defhatI}
    \hat{I}^{N,j}_t := \trace(\rho^{N,j}_t q_{j,t}) = 1 - \trace(\rho^{N,j}_t p_{j,t}) = 1 - (\phi_{t}^{MF,j}, \rho^{N,j}_t \phi_t^{MF,j})_{L^2(\R^3)},
\end{equation}
so that $I^{N,j}(t) = \E[\hat{I}^{N,j}_t]$. Moreover, since the $(\phi^{MF,j})_{j \leq N}$ are i.i.d, we have $\E[\hat{I}^{N,j}_t] = \E[\hat{I}^{N,1}_t]$ for all $j\leq N$, by equality in law. 

\subsection{Proof of Theorem \ref{thm:cvgmf}}

We are now ready to prove Theorem \ref{thm:cvgmf}. Start by observing that, by basic properties of the partial trace:
\begin{equation*}
    \hat{I}^{N,1}_t = 1 - \trace(\rho^N_t p_{1,t}^N).
\end{equation*}
Then, applying Ito's formula,  it holds that
\begin{equation*}
    \mathrm{d}\hat{I}^{N,1}_{t} = -\trace\Big(\d\rho^N_{t}p^N_{1,t}\Big) -\trace\Big( \rho^N_{t}\mathrm{d}p^N_{1,t}\Big)-\trace\Big(\d\langle \rho^N, p^N_{1}\rangle_t\Big).
\end{equation*}
When we develop the right hand side of the previous formula, using equations \eqref{eq:BelavkinNp} and \eqref{eq:BelavkinMF}, three types of terms appear. There are drift terms with a Hamiltonian part, drift terms with the coupling operator and martingale terms. Each type of term is grouped together, and denoted respectively $P^{(1)}$, $P^{(2)}$ and $P^{(3)}$. Moreover, when $j \neq 1$, $L_j$ and $p_1$ commute and thus by some careful manipulations, the corresponding terms vanish. At the end, $P^{(2)}$ only contains terms with the coupling operator $L_1$. Full details of the calculations can be found in the proof of Theorem $3.1$ in \cite{kolokoltsov2021law}. Hence, we have
\begin{equation}
    \label{eq:itohatI}
    \mathrm{d}\hat{I}^{N,1}_{t} = (P_t^{(1)} + P_t^{(2)})\mathrm{d}t + \sum_{l=1}^{N}P_t^{(3,l)}\mathrm{d}B_{t}^{l}.
\end{equation}
To deal with martingale terms, we will use the following lemma that shows that the terms $(P^{(3,l)})_l$ are bounded by elements in $L^2(\Omega, L^{\infty}([0,T_0], \R))$.
\begin{lemma}
\label{lem:boundmartP3}
Let $N \geq 2$ and $ l \leq N$. Then, there exists a deterministic constant $C_l > 0$ which depends only of $\| L \|$ such that almost surely
\begin{equation*}
    \sup_{t \leq T_0 } | P^{(3,l)}_t| \leq C_l.
\end{equation*}
\end{lemma}
The proof of this lemma can be found in Appendix \ref{app:boundP3}. Therefore, as the martingale part is square integrable, by taking the expectation we have 
\begin{equation*}
    \d I^{N,1}(t) = \E[P_t^{(1)}]\d t  + \E [P_t^{(2)}] \d t,
\end{equation*}
and hence the terms $(P^{(3,l)})_l$ do not play any important part in the following, so we do not give their explicit form in the main body of the article. The two terms in the drift are given by the following:
\begin{equation}
    \label{eq:defP1}
    \begin{aligned}
     P_t^{(1)} =
\trace\left( \i \left[\frac{1}{N}\sum_{j \neq 1}V_{1,j} - V_1^{\xi_t}, I - p^N_{1,t}\right]\rho_t^{N}\right),
    \end{aligned}
\end{equation}
where $V_{1,j}$ is the multiplication operator by $V(x_1-x_j)$ for $x_1, x_j \in \R^3$ and
\begin{equation*}
\begin{aligned}
 P_t^{(2)} &= - \trace\big(p^N_{1,t} L_1\rho^N_tL_1^* + p^N_{1,t} L_1^* \rho^N_t L_1 + p^N_{1,t} L_1^* {\rho}^{N}_tL_1^* + p^N_{1,t} L_1 + {\rho}^{N}_tL_1\big) \\
&+\trace\big(p^N_{1,t}\rho^N_t L_1 + p^N_{1,t} L_1^* \rho^N_t\big)\trace\big(p^N_{1,t}(L_1^*+L_1)\big) \\
&+ \trace\big(p^N_{1,t}\rho^N_tL_1^* + p^N_{1,t} L_1\rho^N_t\big)\trace\big(\rho^N_t(L_1^*+L_1)\big)\\
&- \trace\big(\rho^N_tp^N_{1,t}\big)\trace\big(\rho^N_t(L_1^* + L_1)\big)\trace\big(p^N_{1,t}(L_1+L_1^*)\big).
\end{aligned}
\end{equation*}
Then, by using the basic property of the partial trace, since $p^N_1$ and $L_1$ acts only on the coordinates $x_1$, this term can be rewritten as
\begin{equation*}
\begin{aligned}
 P_t^{(2)} &= - \trace\big(p_{1,t} L\rho^{N,1}_t L^* + p_{1,t} L^* \rho^{N,1}_t L + p_{1,t} L^* {\rho}^{N,1}_t L^* + p_{1,t} L_1 + {\rho}^{N,1}_t L\big) \\
&+\trace\big(p_{1,t}\rho^{N,1}_t L + p_{1,t} L^* \rho^{N,1}_t\big)\trace\big(p_{1,t}(L^*+L)\big) \\
&+ \trace\big(p_{1,t}\rho^{N,1}_t L^* + p_{1,t} L\rho^{N,1}_t\big)\trace\big(\rho^{N,1}_t(L^*+L)\big)\\
&- \trace\big(\rho^{N,1}_t p_{1,t}\big)\trace\big(\rho^{N,1}_t(L^* + L)\big) \trace\big(p_{1,t}(L+L^*)\big).
\end{aligned}
\end{equation*}
Notice that, if $\rho^N$ is given by \eqref{eq:defrhon}, then $\rho^{N,1}$ is a density operator. The term with $P^{(2)}$ is tackled by a lemma first proved in \cite[Lemma 6.1]{kolokoltsov2021law}, which is recalled in Appendix \ref{sec:appendix-techn} in Proposition \ref{prop:lemKol} (where an alternative proof is given). A direct application of this proposition leads to the following lemma.
\begin{lemma}
Let $t \leq T_0$ and $N \in \N^*$. Then, there exists a universal constant $C > 0$ such that
\begin{equation*}
    |P_t^{(2)}| \leq C \| L \|^2 (1 - \trace \rho^{N,1}_{t} p_{1,t}).
\end{equation*}
\end{lemma}

Thus, there exists $C_1 >0$ depending only of $\| L \|$ such that:
\begin{equation}
    \label{eq:ineqcvgP2}
    \E P_t^{(2)} \leq C_1(\| L \|) I^{N,1}_t.
\end{equation}

It remains to estimate the term $P_t^{(1)}$. Similarly as in \cite{kolokoltsov2022quantum}, let us define the following random functions:
\begin{equation}
    \label{eq:defdeltaf}
    \delta^N_t(x) = \Frac{1}{N-1} \sum_{j=2}^N |\phi^{MF,j}_t(x)|^2 - \xi_t(x),
\end{equation}
where $\xi_t(x) = \E |\phi^{MF}_t(x)|^2=\E(p_{1,t}(x,x))$ for all $t \leq T_0$ and $x \in \R^3$. Some estimates will be needed on this function in the proof, they are summarised in the following lemma:

\begin{lemma}
\label{lem:bdelta}
Let $T_0 >0$ and $N \geq 2$. Let $\delta^N$ be the random function defined in \eqref{eq:defdeltaf}. Then for all $t\leq T_0$ we have
\begin{equation}\label{eq:bounddeltaL1}
    \E |\delta_t^N|_{L^1_x}  \leq 2,
\end{equation}
and there exists $C(T_0) >0$ depending on $T_0$ but independent of $N$, such that for all $t \leq T_0$
\begin{equation}
    \label{eq:bounddeltaL2}
    \E |\delta_t^N|_{L^2_x} \leq \Frac{C(T_0)}{\sqrt{N-1}}.
\end{equation}
\end{lemma}

\begin{proof}
The first bound follows directly from the triangular inequality:
\begin{equation*}
    \E |\delta_t^N|_{L^1_x}  \leq \Frac{1}{N-1} \sum_{j=2}^N \E \left[|\phi^{MF,j}_t|^2_{L^2_x}\right] + |\xi_t|_{L^1_x} \leq 2
\end{equation*}
where we used the conservation of the $L^2_x$-norm of $\phi^{MF,j}$ and of the $L^1_x$-norm of $\xi$.

For the second bound, first notice that
\begin{equation}
    \label{eq:bounddeltaint1}
    \E |\delta_t^N|_{L^2_x}^2 \leq \Frac{1}{N-1}\E |\phi^{MF}_t|^4_{L^4_x}
\end{equation}
since, for all $t \leq T_0$, the $(|\phi_t^{MF,j}|^2)_j$ is a sequence of i.i.d. random functions of mean $\xi_t$. We recall the following sub-case of the Gagliardo-Nirenberg's inequality
\begin{equation*}
    |u|_{L^4_x} \leq |u|_{L^2_x}^{1/4}  |\nabla u|_{L^2_x}^{3/4}.
\end{equation*}
Thanks to the above inequality and the almost sure conservation of the $L^2_x$-norm of $\phi^{MF}$:
\begin{equation}
    \label{eq:bounddeltaint2}
    \E |\phi^{MF}_t|^4_{L^4_x} \leq \E |\phi^{MF}_t|^{3}_{H^1_x}.
\end{equation}
By assumption, there exists $p \geq 4$ such that $\phi_0 \in L^p(\Omega, H^1_x)$. Thus by Proposition \ref{prop:borneh1}, there exists $\hat{C}>0$, which depends on $T_0$ but is independent of $N$, such that 
\begin{equation}
    \label{eq:bounddeltaint3}
    \E |\phi^{MF}_t|^{3}_{H^1_x}\leq (\E |\phi^{MF}_t|^{p}_{H^1_x})^{3/p} \leq (\hat{C} \E |\phi_0|^{p}_{H^1_x})^{3/p} = C^2,
\end{equation}
where we used Hölder's inequality for the first inequality. Collecting \eqref{eq:bounddeltaint1}, \eqref{eq:bounddeltaint2} and \eqref{eq:bounddeltaint3} yields by Cauchy-Schwarz inequality
\begin{equation*}
    \E |\delta_t^N|_{L^2_x} \leq \Frac{C(T_0)}{\sqrt{N-1}},
\end{equation*}
which concludes the proof of the lemma.
\end{proof}

We now go back to the proof of Theorem \ref{thm:cvgmf}. The quantity $P_t^{(1)}$, given by \eqref{eq:defP1}, can break down as follows
\begin{equation}
    \label{eq:decompoP1}
    \begin{aligned}
     |P_t^{(1)}| &\leq \Frac{2}{N} \left|\trace \left( \sum_{j=2}^N(V_{1,j} - V_1^{|\phi^{MF,j}_t|^2})q_{j,t}^N \rho^N_t \right)\right|  + 2 |\trace(V^{\delta^N_t} q_{1,t}^N \rho^N_t)| \\
     &+ \Frac{2}{N} |Tr (V_1^{\xi_t} q_{1,t}^N \rho^N_t)|
\\
&= 2( I + II + III),
    \end{aligned}
\end{equation}
where the projectors $(q_{j,t})_{j \leq N}$ are defined by \eqref{eq:defqj}.
For the term $III$, notice that by trace inequalities:
\begin{equation*}
    III \leq \Frac{1}{N}\| V_1^{\xi_t} \| \|q_{1,t}^N \rho^N_t \|_{\Lc^1} \leq \Frac{1}{N}\| V_1^{\xi_t} \| \| q_{1,t}^N \| \| \rho^N_t\|_{\Lc^1}.
\end{equation*}
Moreover, since $V_1^{\xi_t}$ is a multiplication operator, observe that by basic convolution inequality
\begin{equation*}
    \| V_1^{\xi_t} \| = |V \star \xi_t|_{L^{\infty}_x} \leq |V|_{L^{\infty}_x} |\xi_t|_{L^1_x}.
\end{equation*}
And finally, recall that $|\xi_t|_{L^1_x} = 1$ for all $t \leq T_0$ since the $L^2_x$ norm is preserved for the solution of \eqref{eq:mfeq}.
Thus, by using the bounds on the norms of $q_{j,t}$ and $\rho^N_t$, we obtain
\begin{equation}
    \label{eq:ineqcvgIII}
    III \leq \Frac{1}{N}|V|_{L^{\infty}_x}.
\end{equation}

Now, for the term $II$, we take a function $\theta \in C^{\infty}(\R_+, \R_+)$ with compact support, bounded by $1$ and such that
\begin{equation*}
    \theta(x) = 0, \text{ for } x\geq 2 \text{ and } \theta(x) = 1, \text{ for } x\leq 1.
\end{equation*} 
For $M_N>0$ that may depend on $N$, set:
\begin{equation}
    U_{M_N}(x) = \theta\left(\Frac{|x|}{M_N}\right). 
\end{equation}
In a similar way as for $III$, we obtain
\begin{equation*}
    II \leq \|V^{\delta^N_t} \| = |V \star \delta^N_t|_{L^{\infty}_x} \leq |(U_{M_N} V) \star \delta^N_t|_{L^{\infty}_x} + |((1-U_{M_N})V) \star \delta^N_t|_{L^{\infty}_x} = II_1 + II_2
\end{equation*}
Then, by a basic convolution inequality, we find that there exists $C_2 >0$ independent of $N$ such that
\begin{eqnarray*}
II_1 \leq |U_{M_N} V|_{L^2_x} |\delta^N_t|_{L^2_x} \leq C_2 M_N^3 |V|_{L^{\infty}_x} |\delta^N_t|_{L^2_x}.
\end{eqnarray*}
Moreover, again by a basic convolution inequality, it holds that
\begin{equation*}
    II_2 \leq \sup_{|x|>M_N}|V(x)| \; |\delta_t^N|_{L^1_x}.
\end{equation*}
Collecting both inequalities and taking the expectations yield
\begin{equation*}
    \E II \leq C_2 M_N^3 |V|_{L^{\infty}_x} \E |\delta^N_t|_{L^2_x} + \sup_{|x|>M_N}|V(x)|\;  \E |\delta_t^N|_{L^1_x}.
\end{equation*}
Then, by the estimates \eqref{eq:bounddeltaL1} and \eqref{eq:bounddeltaL2} of Lemma \ref{lem:bdelta}
\begin{equation}
\label{eq:ineqcvgII}
    \E II \leq C_2 \Frac{M_N^3}{\sqrt{N-1}} + 2  \sup_{|x|>M_N}|V(x)|. 
\end{equation}
By taking, for instance, $\eta \in (0, 1/2)$ and
\begin{equation}
    \label{eq:defMN}
    M_N = (N-1)^{(1/2 - \eta)/3}
\end{equation}
the left hand side in \eqref{eq:ineqcvgII} goes to $0$ as $N$ goes to infinity.

Finally, the term $I$ is the one with the most involved estimations. First, recall that $\rho^N$ is the projector on the wave function $\Psi^N$, hence $I$ can be rewritten as
\begin{equation*}
    I = \Frac{1}{N}|(\Psi^N_t, \sum_{j=2}^N (V_{1,j}-V^{|\phi^{MF,j}_t|^2}_1) q_{1,t}^N \Psi^N_t)_{L^2_x}|.
\end{equation*}
Second, notice that for all $j \leq N$:
\begin{equation}
\label{eq:relationpq}
    I_d^{\otimes N} = p_{j,t}^N + q_{j,t}^N,
\end{equation}
which leads to the following decomposition:
\begin{equation*}
    I \leq \frac{1}{N}\sum_{j=2}^N |(\Psi^N_t, q_j^N (V_{1,j}-V^{|\phi^{MF,j}_t|^2}_1) q_1^N \Psi^N)_{L^2_x}| + \frac{1}{N}\sum_{j=2}^N |(\Psi^N_t, p_j^N (V_{1,j}-V^{|\phi^{MF,j}_t|^2}_1) q_1^N \Psi^N)_{L^2_x}|.
\end{equation*}
For the first term in the right hand side, by Cauchy-Schwarz inequality,
\begin{eqnarray*}
|(\Psi^N_t, q_j^N (V_{1,j}-V^{|\phi^{MF,j}_t|^2}_1) q_1^N \Psi^N)_{L^2_x}| &\leq& |q_{j,t}^N \Psi^N_t|_{L^2_x} \| V_{1,j}-V^{|\phi^{MF,j}_t|^2}_1\| \; |q_{1,t}^N \Psi^N_t|_{L^2_x} \\
&\leq& \sqrt{\hat{I}^{N,j}_{t}}\sqrt{\hat{I}^{N,1}_{t}} (|V|_{L^{\infty}_x} + |V \star |\phi^{MF,j}_t|^2|_{L^{\infty}_x}) \\
&\leq& \Frac{1}{2}(\hat{I}^{N,j}_{t} + \hat{I}^{N,1}_{t})(|V|_{L^{\infty}_x} + |V|_{L^{\infty}_x} |\phi^{MF,j}_t|^2_{L^{2}_x}) \\
&\leq& |V|_{L^{\infty}_x} (\hat{I}^{N,j}_{t} + \hat{I}^{N,1}_{t}),
\end{eqnarray*}
where we used the the conservation of the $L^2_x$-norm of $\phi^{MF}$. We also used that, by definition of $\rho^N$ with $\Psi^N$,
\begin{equation*}
    |q_{j,t}^N \Psi^N_t|_{L^2_x}^2 = (\Psi^N_t, q_{j,t}^N \Psi^N_t)_{L^2_x} = \trace q_{j,t}^N \rho^N_t = \hat{I}^{N,j}_t.
\end{equation*}
Hence we obtain
\begin{equation}
    \label{eq:ineqIp1}
    \frac{1}{N}\sum_{j=2}^N |(\Psi^N_t, q_j^N (V_{1,j}-V^{|\phi^{MF,j}_t|^2}_1) q_1^N \Psi^N)_{L^2_x}| \leq |V|_{L^{\infty}_x} (\Frac{1}{N}\sum_{j=2}^N \hat{I}^{N,j}_{t} + \hat{I}^{N,1}_{t}).
\end{equation}
By taking the expectation and the mean equality of the $\hat{I}^{N,j}$ noticed before, the above inequality becomes
\begin{equation}
    \label{eq:boundIfirst}
    \frac{1}{N}\sum_{j=2}^N \E (\Psi^N_t, q_j^N (V_{1,j}-V^{|\phi^{MF,j}_t|^2}_1) q_1^N \Psi^N)_{L^2_x} \leq 2 |V|_{L^{\infty}_x} I^{N,1}(t).
\end{equation}

For the second term, we use once again a decomposition using the relationship \eqref{eq:relationpq} to obtain
\begin{eqnarray*}
|(\Psi^N_t, p_j^N (V_{1,j}-V^{|\phi^{MF,j}_t|^2}_1) q_1^N \Psi^N)_{L^2_x}| &\leq& |(\Psi^N_t, p_j^N (V_{1,j}-V^{|\phi^{MF,j}_t|^2}_1) p_j^N q_1^N \Psi^N)_{L^2_x}| \\
&+& |(\Psi^N_t, p_j^N (V_{1,j}-V^{|\phi^{MF,j}_t|^2}_1) q_j^N q_1^N \Psi^N)_{L^2_x}| \\
&=& A_j + B_j.
\end{eqnarray*}
By definition of $p_j$ and a computation done in \cite{knowles2010mean} (term $(3.17)$), we remark that
\begin{equation*}
    p_j^NV_{1,j} p^N_j = V_1^{|\phi^{MF,j}_t|^2} p_j^N.
\end{equation*}
Thus, it directly follows that $A_j=0$ for all $j \geq 2$.
For $B_j$, by \eqref{eq:relationpq}, we find the following decomposition
\begin{equation}
    \label{eq:defbj12}
    \begin{aligned}
     B_j &\leq |(\Psi^N_t, q_1^N p_j^N (V_{1,j}-V^{|\phi^{MF,j}_t|^2}_1) q_j^N q_1^N \psi^N)_{L^2_x}| \\
&+ |(\Psi^N_t, p_1^N p_j^N (V_{1,j}-V^{|\phi^{MF,j}_t|^2}_1) q_j^N q_1^N \psi^N)_{L^2_x}| \\
&= B_j^1 + B_j^2.
    \end{aligned}
\end{equation}

For $B^1_j$, by Cauchy-Schwarz and similar manipulations as before
\begin{eqnarray*}
B_j^1 &\leq& |p_j^N q_1^N \Psi^N|_{L^2_x} |q_j^N q_1^N \Psi^N|_{L^2_x} \| V_{1,j}-V^{|\phi^{MF,j}_t|^2}_1\| \\
&\leq& 2 |V|_{L^{\infty}_x} | q_1^N \Psi^N|_{L^2_x}^2.
\end{eqnarray*}
Therefore, we have the following bound:
\begin{equation}
    \label{eq:ineqB1}
    B_j^1 \leq 2 |V|_{L^{\infty}_x}\hat{I}^{N,1}_{t}.
\end{equation}
The term $B^2_j$ needs the tools developed in \cite{pickl2011simple,knowles2010mean}, and then extended to a stochastic regime in \cite{kolokoltsov2022quantum}. The bound on $B^2_j$ is
\begin{equation}
    \label{eq:ineqB2}
    \E B^2_j \leq 2 |V|_{L^{\infty}_x}(I^{N,1}_t + \Frac{1}{N}).
\end{equation}
The proof of the above inequality is briefly recalled in Appendix \ref{app:boundKol}. Combining \eqref{eq:defbj12}, \eqref{eq:ineqB1} and \eqref{eq:ineqB2}, it holds that
\begin{equation}
    \label{eq:boundIsecond}
    \Frac{1}{N}\sum_{j=2}^N \E (\Psi^N_t, p_j^N (V_{1,j}-V^{|\phi^{MF,j}_t|^2}_1) q_1^N \Psi^N)_{L^2_x} \leq 4 |V|_{L^{\infty}_x} (I^{N,1}(t) + \Frac{1}{N}).
\end{equation}

By collecting the results \eqref{eq:boundIfirst} and \eqref{eq:boundIsecond}, we find that there exists a universal constant $C_3 >0$ such that
\begin{equation}
    \label{eq:ineqcvgI}
    \E I \leq C_3 |V|_{L^{\infty}_x} \big(I^{N,1}(t) + \Frac{1}{N}\big).
\end{equation}
Hence, by \eqref{eq:decompoP1} and then combining \eqref{eq:ineqcvgIII}, \eqref{eq:ineqcvgII} and \eqref{eq:ineqcvgI}, the bound on $P^{(1)}$ becomes
\begin{equation}
    \label{eq:ineqcvgP1}
    \E P^{(1)}_t \leq C_4 \big( |V|_{L^{\infty}_x} I^{N,1}(t) + \Frac{1}{N}|V|_{L^{\infty}_x} + \Frac{M_N^3}{\sqrt{N-1}} +  \sup_{|x|>M_N}|V(x)| \big),
\end{equation}
where $C_4 >0$ is a universal constant. Finally, collecting \eqref{eq:ineqcvgP2} and\eqref{eq:ineqcvgP1} with $\eta \in (0, 1/2)$ and $M_N$ given by \eqref{eq:defMN} for such $\eta$, the bound on $I^{N,1}_t$ becomes
\begin{equation}
    I^{N,1}(t)' \leq (C_4 |V|_{L^{\infty}_x} + C_1)I^{N,1}(t) + C_4 \big(\Frac{1}{N}|V|_{L^{\infty}_x} + (N-1)^{-\eta} + \sup_{|x|>M_N}|V(x)| \big).
\end{equation}
Then, by applying Gronwall's lemma we find
\begin{equation}
    \sup_{t \leq T_0} I^{N,1}(t) \leq C_4 \left( \Frac{1}{N}|V|_{L^{\infty}_x} + (N-1)^{-\eta} +  \sup_{|x|>M_N}|V(x)|\right) e^{(C_4 |V|_{L^{\infty}_x} + C_1) T_0}
\end{equation}
which goes to $0$ as $N$ goes to infinity, since $M_N$ goes to infinity. This concludes the proof of Theorem \ref{thm:cvgmf}.
\hfill
$\square$

\begin{appendix}

\section{Some technical estimates}
\label{sec:appendix-estF}
In this section, we derive some estimates on the two intermediate functions used in the proofs. Recall that $F_1$ and $F_2$ are defined in the following way:
\begin{equation*}
    F_1(X) = L ^* L -2 \langle L \rangle_{X} L +\langle L \rangle_{X}^2 I_d,
\end{equation*}
and 
\begin{equation*}
    F_2(X) = L - \langle L \rangle_{X}I_d.
\end{equation*}

\begin{proof}[Proof of Lemma \ref{lem:estimateF}]
For the first estimate, note that
\begin{eqnarray*}
|F_1(X)X|_{L^2_x} &\leq& |L^* L X|_{L^2_x} + 2 |\langle L \rangle_{X}| |L X|_{L^2_x} + \langle L \rangle_{X}^2|X| \\
&\leq& 2R \| L \|^2 (1+ 4 R^2)^2.
\end{eqnarray*}
The second estimate needs more careful manipulations. First, decompose the difference in the following way:
\begin{eqnarray*}
|F_1(X)X - F_1(Y)Y|_{L^2_x} &\leq& \| L \|^2 |X-Y|_{L^2_x} + 2 \| L \| |\langle L \rangle_{X}X -\langle L \rangle_{Y}Y|_{L^2_x}\\
&+&|\langle L \rangle_{X}^2X -\langle L \rangle_{Y}^2Y|_{L^2_x} \\
&=& \| L \|^2 |X-Y|_{L^2_x} + 2 \| L \| I + II.
\end{eqnarray*}
$I$ is estimated thanks to Lemma \ref{lem:estimateavg}:
\begin{eqnarray*}
I &\leq& |\langle L \rangle_{X} - \langle L \rangle_{Y}| |X|_{L^2_x} + |\langle L \rangle_{Y}| |X-Y|_{L^2_x} \\
&\leq& 12 R^2 \| L \| |X-Y|_{L^2_x}. 
\end{eqnarray*}
As for $II$, by using once again the inequalities of Lemma \ref{lem:estimateavg}:
\begin{eqnarray*}
II &\leq& |\langle L \rangle_{X}|^2 |X-Y|_{L^2_x}  + |\langle L \rangle_{X}^2 - \langle L \rangle_{Y}^2| |Y|_{L^2_x} \\
&\leq& (4 R^2 \| L \|)^2 |X-Y|_{L^2_x} + |\langle L \rangle_{X} - \langle L \rangle_{Y}| |\langle L \rangle_{X} + \langle L \rangle_{Y}| 2R \\
&\leq& 80 R^4 \| L\|^2 |X-Y|_{L^2_x}.
\end{eqnarray*}
Thus, collecting the results leads to \eqref{eq:estimatef12}
Estimate \eqref{eq:estimatef21} is a direct consequence of the triangular inequality.
Finally, \eqref{eq:estimatef22} comes from a similar decomposition as in the beginning of the proof of \eqref{eq:estimatef12}. More precisely:
\begin{eqnarray*}
|F_2(X) X - F_2(Y)Y|_{L^2_x} &\leq& \| L \| |X-Y|_{L^2_x} + |\langle L \rangle_{X} - \langle L \rangle_{Y}| |X|_{L^2_x} + |\langle L \rangle_{Y}| |X-Y|_{L^2_x} \\
&\leq&  (1 + 12R^2)\| L \| |X-Y|_{L^2_x},
\end{eqnarray*}
which concludes the proof of the lemma.
\end{proof}

With additional assumptions on $L$, we can obtain estimates on $\nabla F_1$ and $\nabla F_2$, which are stated in Lemma \ref{lem:estimFh1}.

\begin{proof}[Proof of Lemma \ref{lem:estimFh1}]
For the estimate \eqref{eq:estimf1h1} start by decomposing in the following way:
\begin{equation*}
\begin{aligned}
|\nabla F_1(X) X|_{L^2_x} &\leq |F_1(X) \nabla X|_{L^2_x} + |[\nabla, F_1(X)] X|_{L^2_x} \\
&= I + II.
\end{aligned}
\end{equation*}
For $I$, by performing exactly the same manipulations as in the proof of the first estimate of Lemma \ref{lem:estimateF}, we get
\begin{equation}\label{eq:estimf1h11}
    I \leq  (1 + 4R^2)^2 \| L\|^2 |\nabla X|_{L^2_x}.
\end{equation}
For $II$, start by noticing that the commutator has the following form 
\begin{equation*}
    [\nabla, F_1(X)] = [\nabla, L^* L] -2 \langle L \rangle_{X} [\nabla, L].
\end{equation*}
Therefore, by Lemma \ref{lem:estimateavg}
\begin{equation}\label{eq:estimf1h12}
    \begin{aligned}
    II \leq (\|[\nabla, L^* L] \|_{\Lc(H^1_x, L^2_x)} + 8 R^2 \| L \| \|[\nabla, L] \|_{\Lc(H^1_x, L^2_x)}) |X|_{H^1_x}
    \end{aligned}.
\end{equation}
Combining \eqref{eq:estimf1h11} and \eqref{eq:estimf1h12} yields estimate \eqref{eq:estimf1h1}.

For the estimate \eqref{eq:estimf2h1}, with the same manipulations, we find the following inequality:
\begin{equation}
    |\nabla F_2(X) X|_{L^2_x} \leq (\| L \| + 4 R^2\| L \| + \|[\nabla, L] \|_{\Lc(H^1_x, L^2_x)}) |X|_{H^1_x},
\end{equation}
which is the desired result.
\end{proof}

\section{A Stochastic Differential Equation}
\label{sec:appendix-SDE}

In this section, we solve the SDE arising from studying the $L^2_x$-norm of \eqref{eq:trunc}.

\begin{lemma}
\label{lem:SDE}
Assume that an almost surely continuous stochastic process $(M_t)_{t\ge 0}$ satisfies the Itô's dynamics
$$\d M_t = C_t(M_t-1) \d W_t,\quad \forall t>0,$$
where $(C_t)_{t\ge 0}$ is adapted with $|C_t| \le C|M_t|$, almost surely, for some deterministic $C>0$, and $(W_t)_{t\ge 0}$ is a real-valued Brownian motion. If $M_0=1$, then $M_t=1$ for all $t\geq 0$, almost surely.  
\end{lemma}
\begin{proof}
It suffices to adopt the localization argument. Set for $n\ge 1$, 
$$\tau_n:=\inf\{t\ge 0:  |M_t|\ge n \}.$$ 
Then the stopped processes $M^n_t:= M_{t\wedge \tau_n}$, $C^n_t:= C_{t\wedge \tau_n}$ satisfy
$$M^n_t-1 =\int_0^{t\wedge \tau_n} C_u(M_u-1)\d W_u= \int_0^{t\wedge \tau_n} C^n_u(M^n_u-1)\d W_u.$$
Doob's inequality allows to conclude. Namely, for any $T>0$, 
\begin{eqnarray*}
\mathbb E[\sup_{0\leq t\leq T}|M^n_{t}-1|^2] &\leq& 4\mathbb E\left[\int_0^{T\wedge \tau_n} \sup_{0\leq t\leq u}\big(C^n_t(M^n_t-1)\big)^2\d u\right] \\
&\leq& 4C^2(n+1)^2\int_0^{T} \mathbb E[\sup_{0\leq t\leq u}|M^n_{t}-1|^2]\d u.
\end{eqnarray*}
This yields $\sup_{0\leq t\leq T}|M^n_{t}-1|^2=0$ almost surely. Since $M$ is almost-surely continuous, we can complete the localization argument, which concludes the proof.
\end{proof}

\section{Kolokolstov's Technical Lemma}
\label{sec:appendix-techn}

In this section, we give an alternative proof of the technical lemma used in Theorem \ref{thm:cvgmf}, proved first in \cite{kolokoltsov2021law}. In \cite{kolokoltsov2021law}, the proof was based by approximation argument, here we prove the result using only trace inequalities and properties of projectors and positive operators.

In the following, we fix a separable Hilbert space $\Hc$ with an orthonormal basis $(e_j)_{j\in \N}$. We recall the result with the following proposition.

\begin{proposition}\label{prop:lemKol}
Let $L \in \Lc(\Hc)$, $p$ a one-dimensional projector and $\hat{\rho}$ a density operator. Then if $L$ is self-adjoint, there exists a universal constant $C_1 > 0$ such that
\begin{equation}
    \label{eq:lemmKolsadj}
    \begin{aligned}
    &\big|-4 \trace (L \hat{\rho} L p ) + 2\big( \trace(\hat{\rho} L p ) + \trace(L \hat{\rho} p) \big) \big(\trace(L \hat{\rho}) + \trace(Lp) \big)\\
    &- 4 \trace(\hat{\rho} p)\trace(L p)\trace(L\hat{\rho})\big| \leq C_1\| L \|^2 \big(1-\trace(\hat{\rho} p)\big),
    \end{aligned}
\end{equation}
and for a general $L$ there exists a universal constant $C_2 >0$ such that
\begin{equation}
    \label{eq:lemmKolsgene}
    \begin{aligned}
     &\big|-\trace(p L \hat{\rho} L^* + p L^* \hat{\rho} L + p L \hat{\rho} L + p L^* \hat{\rho} L^*)  - \trace(\hat{\rho} p) \trace \big(\hat{\rho}(L+L^*)\big) \trace \big(p (L + L^*)\big) \\
     &+ \trace(p \hat{\rho} L^* + p L \hat{\rho})\trace\big(p(L^* + L)\big) + \trace(p \hat{\rho} L + p L^* \hat{\rho})\trace\big(\hat{\rho}(L^* + L)\big)\big|\\
    &\leq C_2\|L\|^2 \big(1-\trace(\hat{\rho} p)\big).
    \end{aligned}
\end{equation}
\end{proposition}
Before proving this proposition, we start by stating two general results on operators that will play a key role in the proof.

\begin{lemma}\label{lem:prodproj}
Let $p$ be a one-dimensional projector and $A,B$ two bounded operators on $\Hc$, then
\begin{equation*}
    \trace{A p B p} = \trace{A p} \: \trace{B p}.
\end{equation*}
\end{lemma}

\begin{proof}
Since $p$ is a one-dimensional projector, there exists $\phi \in \Hc$ with $|\phi|_{\Hc} = 1$ such that 
\begin{equation*}
    p = (\phi, \cdot) \phi.
\end{equation*}
First, let us show that $\trace{Ap} = (A\phi, \phi)$. By definition of the trace, we have
\begin{equation*}
    \trace{Ap} = \sum_{j=1}^{\infty} (A p e_j, e_j). 
\end{equation*}
Thus, by definition of the projector $p$
\begin{equation*}
    \sum_{j=1}^{\infty} (A p e_j, e_j) = \sum_{j=1}^{\infty} (A \phi, e_j) (\phi, e_j) = (A \phi, \sum_{j=1}^{\infty} (\phi, e_j) e_j).
\end{equation*}
Hence the result follows. Second, to show the lemma we use twice the previous equality
\begin{eqnarray*}
\trace{A p B p} &=& (A p B \phi, \phi) \\
&=& (A (B \phi, \phi) \phi, \phi) \\
&=& (A \phi, \phi)(B \phi, \phi) = \trace{A p} \: \trace{B p}.
\end{eqnarray*}
\end{proof}

\begin{lemma}\label{lem:estimGamA}
Let $\hat{\rho} \in \Lc^1(\Hc)$ be a positive symmetric operator and $A \in \Lc(\Hc)$, then
\begin{equation}
    \| \hat{\rho} A \|^2_2 \leq \trace(\hat{\rho} ) \trace(A^* \hat{\rho} A)
\end{equation}
\end{lemma}

\begin{proof}
By definition of the $\Lc^2$-norm and Parseval's Theorem
\begin{equation*}
    \begin{aligned}
     \| \hat{\rho} A \|^2_2 &= \sum_{j=1}^{\infty}(\hat{\rho} A e_j, \hat{\rho} A e_j )\\
     &= \sum_{j,k=1}^{\infty}(\hat{\rho} A e_j, e_k)^2.
    \end{aligned}
\end{equation*}
Since $\hat{\rho}$ is a positive symmetric operator, Cauchy-Schwartz inequality implies
\begin{equation*}
    (\hat{\rho} A e_j, e_k)^2 \leq (\hat{\rho} A e_j, A e_j) (\hat{\rho} e_k, e_k).
\end{equation*}
Therefore
\begin{equation*}
    \begin{aligned}
     \| \hat{\rho} A \|^2_2 &\leq \sum_{k=1}^{\infty} (\hat{\rho} e_k, e_k) \sum_{j=1}^{\infty} (\hat{\rho} A e_j, A e_j)
     \\&= \trace(\hat{\rho}) \trace(A^* \hat{\rho} A),
    \end{aligned}
\end{equation*}
as required.
\end{proof}

\begin{proof}[Proof of Proposition \ref{prop:lemKol}]
To lighten the computations, we will note
\begin{equation*}
    \alpha := \trace((I-p)\hat{\rho}) = 1 - \trace(\hat{\rho} p).
\end{equation*}

To begin with, we restrict ourselves to the case where $L$ self-adjoint. By rewritting the left hand side of \eqref{eq:lemmKolsadj}, we find
\begin{eqnarray*}
&&-4 \trace (L \hat{\rho} L p ) + 2( \trace(\hat{\rho} L p ) + \trace(L \hat{\rho} p) ) (\trace(L \hat{\rho}) + \trace(Lp) ) - 4 \trace(\hat{\rho} p)\trace(L p)\trace(L\hat{\rho}) \\
   &=& -4 \trace (L \hat{\rho} L p ) + 2( \trace(\hat{\rho} L p ) + \trace(L \hat{\rho} p) )\trace(L p) \\
   &&+ 2(\trace(\hat{\rho} L p ) + \trace(L \hat{\rho} p) - 2 \trace(\hat{\rho} p)\trace(L p)) (\trace(L \hat{\rho}) - \trace(L p) + \trace(Lp))\\
    &=& - 4 [\trace (L \hat{\rho} L p ) - \trace(L \hat{\rho} p )\trace(L p) - \trace(\hat{\rho} L p )\trace(L p) + \trace(\hat{\rho} p)\trace(L p)^2] \\
    &&+ 2[\trace(\hat{\rho} L p ) + \trace(L \hat{\rho} p)- 2 \trace(\hat{\rho} p)\trace(L p)][\trace(L \hat{\rho}) - \trace(L p)]\\
    &=& -4 I + 2 II.
\end{eqnarray*}
For $I$, notice that, thanks to Lemma \ref{lem:prodproj} 
\begin{equation*}
    I = \trace( L (I-p) \hat{\rho} (I-p) L p).
\end{equation*}
Thus, by classic inequalities on the norms we have
\begin{equation*}
    \begin{aligned}
    |I| &\leq \| L p \|_2 \| L (I-p) \hat{\rho} (I-p)\|_2 \\
    &\leq \| L p \|_2 \| L \| \|(I-p) \hat{\rho} (I-p)\|_2 \\
    &\leq \| L p \|_2 \| L \| \|(I-p) \hat{\rho} (I-p)\|_1,
    \end{aligned}
\end{equation*}
where we used $\| \cdot \|_2 \leq \| \cdot \|_1$. Since $\hat{\rho}$ is a positive operator and $I-p$ is self-adjoint, it implies $(I-p) \hat{\rho} (I-p)$ is also a positive operator. Therefore, using $(I-p)^2 = I-p$, we obtain
\begin{equation*}
    \|(I-p) \hat{\rho} (I-p)\|_1 = \trace ((I-p) \hat{\rho}) = \alpha.
\end{equation*}
Moreover, since $p$ is a projector on an element with norm one, $\|p\|_2 = 1$. Thus
\begin{equation*}
    \| L p \|_2 \leq \| L \| \| p \|_2 = \| L \|.
\end{equation*}
Collecting the above inequalities, we find
\begin{equation}
    \label{eq:lemKolestimI}
    |I| \leq \| L \|^2 \alpha.
\end{equation}
For $II$, recall first the following bound, given by Lemma $2.3$ in \cite{knowles2010mean}
\begin{equation}
    \label{eq:picklbound}
    \| \hat{\rho} - p \|_1 \leq 2 \sqrt{2} \sqrt{\alpha}.
\end{equation}
Notice that $II$ can be rewritten as
\begin{equation*}
   \begin{aligned}
    II &= [\trace(L \hat{\rho} p) + \trace( \hat{\rho} L p) - 2 \trace(L p \hat{\rho} p)] [\trace(L(\hat{\rho} - p))]\\
    &=[\trace(p L (\hat{\rho} - \hat{\rho} p)) + \trace( (\hat{\rho} - \hat{\rho} p) L p)]\trace(L(\hat{\rho} - p)),
   \end{aligned}
\end{equation*}
where we used Lemma \ref{lem:prodproj} and some commutations properties in the trace. Then, we have by Lemma \ref{lem:estimGamA}
\begin{equation*}
    \begin{aligned}
     |\trace(p L (\hat{\rho} - \hat{\rho} p))| &\leq \|p L \|_2 \| \hat{\rho} (I - p) \|_2\\
     &\leq \| L \| \big(\trace \hat{\rho} \: \trace((I-p) \hat{\rho} (I-p)) \big)^{1/2}\\
     &\leq \| L \| \sqrt{\alpha},
    \end{aligned}
\end{equation*}
since $\hat{\rho}$ is a density operator. The same bound can be obtained for the second term in the right hand side in the previous expression of $II$. Moreover, by \eqref{eq:picklbound}
\begin{equation*}
     |\trace(L(\hat{\rho} - p)) | \leq \| L \| \| \hat{\rho} - p \|_1 \leq 2 \sqrt{2} \| L \| \sqrt{\alpha}. 
\end{equation*}
Hence, combining the two previous inequalities, we find
\begin{equation}
    \label{eq:lemKolestimII}
    |II| \leq (2 \|L\|\sqrt{\alpha})(2 \sqrt{2}\|L \| \sqrt{\alpha}) = 4\sqrt{2} \|L\|^2 \alpha.
\end{equation}
Collecting \eqref{eq:lemKolestimI} and \eqref{eq:lemKolestimII}, it holds that there exists $C_1 >0$ a universal constant such that
\begin{equation}
    |-4 I + 2 II| \leq C_1 \|L\|^2 \alpha,
\end{equation}
which is the desired inequality if $L$ is self-adjoint.

In the case of a general operator $L$, we decompose $L$ as
\begin{equation*}
    L = L^s + L^a,
\end{equation*}
with $L^s = (L + L^*)/2$ a self-adjoint operator and $L^a = (L - L^*)/2$, skew-adjoint. This enables us to rewrite the part inside the absolute value on the left hand side of \eqref{eq:lemmKolsgene} in the following way
\begin{eqnarray*}
&&-\trace(p L \hat{\rho} L^* + p L^* \hat{\rho} L + p L \hat{\rho} L + p L^* \hat{\rho} L^*) - \trace(\hat{\rho} p) \trace\big(\hat{\rho}(L+L^*)\big) \trace\big(p (L + L^*)\big)\\
     &&+ \trace(p \hat{\rho} L + p L^* \hat{\rho})\trace\big(\hat{\rho}(L^* + L)\big) + \trace(p \hat{\rho} L^* + p L \hat{\rho})\trace\big(p(L^* + L)\big)\\
     &=&2 \trace\big(p[\hat{\rho}, L^a]\big) \big(\trace(\hat{\rho} L^s) - \trace(p L^s)\big)- 4 \trace(\hat{\rho} p)\trace(L^s p)\trace(L^s\hat{\rho})\\
     &&  -4 \trace (L^s \hat{\rho} L^s p ) + 2\big( \trace(\hat{\rho} L^s p ) + \trace(L^s \hat{\rho} p) \big) \big(\trace(L^s \hat{\rho}) + \trace(L^s p) \big)  \\
     &=& 2II + I.
\end{eqnarray*}
Note that $I$ corresponds to the left hand side of \eqref{eq:lemmKolsadj} with $L$ replaced with $L^s$. Thus,
\begin{equation}
    \label{eq:boundIkolgene}
    |I| \leq C_1 \| L^s \| \alpha \leq C_1 \| L \| \alpha.
\end{equation}
For $II$, recall that by \eqref{eq:picklbound}
\begin{equation}
    \label{eq:boundII1kolgene}
    |\trace(\hat{\rho} L^s) - \trace(p L^s)| \leq 2 \sqrt{2} \| L \| \sqrt{\alpha}.
\end{equation}
Moreover, by using some basic computations
\begin{equation*}
    \begin{aligned}
     |\trace(p[\hat{\rho}, L^a])| &= |\trace(p \hat{\rho} L^a) - \trace(\hat{\rho} p L^a p) + \trace(\hat{\rho} p L^a p) - \trace(\hat{\rho} p L^a)| \\
     &\leq |\trace(\hat{\rho}(I -p)L^a p)| + |\trace((I -p) \hat{\rho} p L^a)| \\
     &\leq \| L^a p \|_2(\| \hat{\rho} (I-p) \|_2 + \| (I-p) \hat{\rho} \|_2).
    \end{aligned}
\end{equation*}
Thus, applying Lemma \ref{lem:estimGamA}, we find
\begin{equation}
    \label{eq:boundII2kolgene}
    |\trace(p[\hat{\rho}, L^a])| \leq 2 \| L \| \sqrt{\alpha}.
\end{equation}
Finally, collecting \eqref{eq:boundIkolgene}-\eqref{eq:boundII2kolgene}, we find that 
\begin{equation*}
    |I + 2II| \leq(C_1 + 4)\|L\|^2 \alpha,
\end{equation*}
as required.
\end{proof}

\begin{remark} It appears that, in an infinite dimensional setting, the condition of $p$ being a rank one projector is a critical assumption in the proof. The proof given in \cite{chalal2023mean} for the finite dimensional setting uses the equivalence of norms, which does not hold in infinite dimension.

\end{remark}

\section{Technical bounds for the proof of Theorem \ref{thm:cvgmf}}

This section is devoted to the proof of two technical estimates needed in the proof of Theorem \ref{thm:cvgmf}. We recall that the projectors $p_j$ are defined by \eqref{eq:defpj} and $\rho^N_t$ by \eqref{eq:defrhon}. 

\subsection{Estimate of the martingale term}
\label{app:boundP3}

The main objective of this subsection is to prove the lemma used to bound the martingale part in the proof of Theorem \ref{thm:cvgmf}. Recall that we denote by $P^{(3,l)}$ the term inside the stochastic integral with respect to the Brownian motion $B^l$ in the expression \eqref{eq:itohatI}. In the proof, we will use various well-known trace norm inequalities, the interested reader can find them and their proof in Section $3.1.3.$ in \cite{kukush2019}.

\begin{proof}[Proof of Lemma \ref{lem:boundmartP3}]
Almost surely, the trace norm of the operators $(p_{1,t})_{j \leq N}$ and $\rho^N_t$ is preserved equal to one for all time $t \leq T_0$. We therefore fix $\omega$ in the set of measure $1$ where all these norms are preserved. We only give the full expression of $P^{(3,l)}$ for $ l \geq 2$. For $t \leq T_0$
\begin{equation}
    \label{eq:expP3l}
    P_t^{(3,l)} = \trace\big(p_{1,t}^N L_l \rho^N_t + p_{1,t}^N \rho^N_t L_l^* - \trace\big( (L_l + L_l^*)\rho^N_t\big)p_{1,t}^N\rho^N_t \big).
\end{equation}
Hence, by the triangular inequality
\begin{equation*}
    \begin{aligned}
     |P_t^{(3,l)}| &\leq |\trace(p_{1,t}^N L_l \rho^N_t)| + |\trace(p_{1,t}^N \rho^N_t L_l^*)| + |\trace\big( (L_l + L_l^*)\rho^N_t\big) \: \trace (p_{1,t}^N\rho^N_t)| \\
     &\leq \| p_{1,t}^N L_l \rho^N_t \|_1 + \| p_{1,t}^N \rho^N_t L_l^* \|_1 + \| (L_l + L_l^*)\rho^N_t \|_1 \| p_{1,t}^N\rho^N_t \|_1\\
     &= I + II + III.
    \end{aligned}
\end{equation*}
For $I$, notice that by trace and operator norm inequalities
\begin{equation}
    \label{eq:martboundtraceI}
    I \leq \| p_{1,t}^N \| \|L \rho^N_t \|_1 \leq \| L \| \| p_{1,t}^N \| \| \rho^N_t \|_1 \leq \| L \|,
\end{equation}
where we used that the operator norm of $p_{1,t}^N$ and the trace norm of $\rho^N_t$ are equal to one. Similarly for $II$, we find
\begin{equation}
        \label{eq:martboundtraceII}
    II \leq \| L \|.
\end{equation}
For $III$, we obtain
\begin{equation}
\label{eq:martboundtraceIII}
    III \leq 2 \| L \| \| \rho^N_t \|_1 \| p_{1,t}^N \| \| \rho^N_t\|_1 \leq 2 \| L \|.
\end{equation}
Thus, collecting \eqref{eq:martboundtraceI}-\eqref{eq:martboundtraceIII} and taking the supremum in time yields
\begin{equation*}
    \sup_{t \leq T_0} |P_t^{(3,l)}| \leq 4 \| L \|,
\end{equation*}
as required. For $l=1$, $P^{(3,1)}$ is the sum of a term of the form \eqref{eq:expP3l} and another similar term where $p_{1,t}^N$ and $\rho^N_t$ are exchanged. Therefore, a similar bound can be obtained with exactly the same algebraic manipulations.
\end{proof}

\subsection{Bound on $B^2_j$}
\label{app:boundKol}

In this subsection, we briefly recall the main arguments of the proof for the bound of $B^2_j$. This is the term that needs more subtle control, and for which new tools have been introduced first in \cite{pickl2011simple, knowles2010mean}. The result for the stochastic setting, obtained in the proof of Proposition A$.2.$ in \cite{kolokoltsov2025quantum}, is stated in the following lemma.
\begin{lemma}\label{lem:boundB2}
With the set of assumptions of Theorem \ref{thm:cvgmf}, for all $t \leq T$
\begin{equation*}
    \E B^2_j \leq 2 |V|_{L^{\infty}_x}(I^{N,1}_t + \Frac{1}{N}),
\end{equation*}
where $B^2_j$ is defined in \eqref{eq:defbj12} and $\hat{I}^{N,1}$ by \eqref{eq:defhatI}.
\end{lemma}

In the following, for the sake of simplicity, we omit the index $t$ of all the projectors. First, define the following operators for $k \leq N$:
\begin{equation*}
    P_k^N = \sum_{a\in \{0,1\}^{N}, \sum a_j = k } \Prod_{j=1}^N (p_j^N)^{1-a_j} (q_j^N)^{a_j},
\end{equation*}
where $q_j$ is the orthogonal projector of $p_j$ defined by \eqref{eq:defqj}. Then define the average operator of the projectors $q_j^N$:
\begin{equation*}
    M_N := \frac{1}{N}\sum_{j=1}^N q_j^N.
\end{equation*}
The first point is to note the following relationship between $M_N$ and the projectors $P_k$ (see for instance relation $(3.11)$ in \cite{knowles2010mean}):
\begin{equation*}
    M_N = \sum_{k=1}^N\frac{k}{N}P_k^N.
\end{equation*}
Next, for a function $f : \{0, \cdots, N\} \mapsto \C $, define the following operator
\begin{equation*}
    \hat{f} := \sum_{k=1}^N f(k) P_k.
\end{equation*}
From this definition, we can rewrite $M_N$ as
\begin{equation*}
    M_N = \widehat{m_N},
\end{equation*}
where $m_N$ is defined by $m_N(k) := k/N$, for $k \leq N$. One of the main argument is that we can define a pseudo-power of $\widehat{m_N}$ by
\begin{equation*}
    \widehat{m_N}^{\alpha} = \sum_{k=1}^N (\frac{k}{N})^{\alpha}P_k,
\end{equation*}
for $\alpha$ a number. Then, one can show that, for all $j \leq N$ and $\alpha \neq 0$,
\begin{equation*}
    \widehat{m_N}^{\alpha}\widehat{m_N}^{-\alpha}q_j = q_j,
\end{equation*}
where the proof can be found in the proof of Proposition A$.2.$ in \cite{kolokoltsov2022quantum}. The usefulness of these notations appears in a key result, stated in the following lemma, that allows to do some pseudo-permutation tricks in the proof:

\begin{lemma}\label{lem:permtrick}
Let $Q_r = \#_1 \cdots \#_r$, where $\#_i$ is either $p_i$ or $q_i$. Then we have, for all $k, r \leq N$
\begin{equation*}
     Q_r P_k =P_k Q_r.
\end{equation*}
Moreover, let $Q^1_r$ and $Q^2_r$ be two such operators with respectively $n_1$ and $n_2$ factors $q$ and $A_r \in \Lc(L^2(\R^3)^{\otimes r})$. Then
\begin{equation*}
    Q^1_r A_r \hat{f} Q^2_r = Q^1_r \widehat{\tau_n f} A_r Q^2_r
\end{equation*}
where $\tau_n f(k) = f(k+n)$ and $n=n_2-n_1$.
\end{lemma}

The proof of this lemma in a stochastic regime can be found in \cite[Lemma A.1.]{kolokoltsov2022quantum}. 

We are now equipped to study the last term $B^2_j$. First, notice that the mean field term $V^{|\phi^{MF,j}_t|^2}_1$ acts only on the particle $1$. Hence $p_j^N$ and  $V^{|\phi^{MF,j}_t|^2}_1$ commute, which implies that this term disappears in $B^2_j$ (since $p_j q_j = 0$). Second, by applying twice Cauchy-Schwarz inequality, we obtain
\begin{eqnarray*}
\E B^2_j &=& \E |(\Psi^N_t, p_1^N p_j^N V_{1,j} \widehat{m_N}^{1/2}\widehat{m_N}^{-1/2} q_j^N q_1^N \Psi_t^N)_{L^2_x}| \\
&=& \E |(\Psi^N_t, p_1^N p_j^N \widehat{\tau_2 m_N}^{1/2}V_{1,j} \widehat{m_N}^{-1/2} q_j^N q_1^N \Psi_t^N)_{L^2_x}| \\
&\leq& \E |V_{1,j}\widehat{\tau_2 m_N}^{1/2}p_j^N p_1^N \Psi^N_t|_{L^2_x} |\widehat{m_N}^{-1/2}q_j^N q_1^N \Psi^N_t|_{L^2_x} \\
&\leq&\big(\E |V_{1,j}\widehat{\tau_2 m_N}^{1/2}p_j^N p_1^N \Psi^N_t|_{L^2_x}^2\big)^{1/2}\big(\E |\widehat{m_N}^{-1/2}q_j^N q_1^N \Psi^N_t|_{L^2_x}^2\big)^{1/2}\\
&:=& \sqrt{\alpha} \sqrt{\beta}.
\end{eqnarray*}
By Lemma \ref{lem:permtrick}, one can notice that
\begin{equation*}
    [p_1^N p_j^N, P^N_k] = 0,
\end{equation*}
and thus $p_1^N p_j^N$ commutes with $\widehat{\tau_2 m_N}$. Then, by applying Lemma \ref{lem:permtrick} and the previous remark we have
\begin{equation*}
    \E |V_{1,j}\widehat{\tau_2 m_N}^{1/2}p_1^N p_j^N \Psi^N_t|_{L^2_x}^2 \leq |V|_{L^{\infty}_x}^2 \E |\widehat{\tau_2 m_N}^{1/2}\Psi^N_t|_{L^2_x}^2 = |V|_{L^{\infty}_x}^2\E (\Psi^N_t, \widehat{\tau_2 m_N}\Psi^N_t)_{L^2_x}.
\end{equation*}
Moreover, a quick computation shows that
\begin{equation*}
    \widehat{\tau_2 m_N} = \widehat{m_N} + \Frac{2}{N},
\end{equation*}
since by definition of the operators $P_k^N$
\begin{equation*}
    \sum_{k=0}^N P_k^N = I_d.
\end{equation*}
Therefore, we have
\begin{equation}
    \label{eq:boundB2jI}
    \alpha \leq |V|_{L^{\infty}_x}^2\big(\E (\Psi^N_t, \widehat{m_N}\Psi^N_t)_{L^2_x} + \Frac{2}{N}\big) = |V|_{L^{\infty}_x}^2( I^{N,1}(t) + \Frac{2}{N}),
\end{equation}
where we used once again the fact that the random variables $(q_{j,t}^N)_{j\leq N}$ have the same law.

Finally, the term $\beta$ can be tackled in the same way as in the proof of Proposition A.$2$. in \cite{kolokoltsov2022quantum} (estimation of the term $(91)$) and it yields
\begin{equation}
    \label{eq:boundB2jII}
    \E |\widehat{m_N}^{-1/2}q_j^N q_1^N \Psi^N_t|_{L^2_x}^2 \leq I^{N,1}(t).
\end{equation}
This computations relies on some careful algebraic manipulations and the equality in law of the random variables $(q_{j,t}^N)_{j\leq N}$. Therefore, combining \eqref{eq:boundB2jI} and \eqref{eq:boundB2jII}, we finally obtain
\begin{equation*}
    \begin{aligned}
     \E B^2_j & \leq |V|_{L^{\infty}_x} \big(I^{N,1}(t) + \Frac{2}{N}\big)^{1/2} \big(I^{N,1}(t)\big)^{1/2}\\
     &\leq 2 |V|_{L^{\infty}_x}(I^{N,1}(t) + \Frac{1}{N}),
    \end{aligned}
\end{equation*}
thanks to Young's inequality, and which is the required bound.
\hfill 
$\square$
\end{appendix}

\section*{Acknowledgments}
G. Guo acknowledges financial support Bourse by \emph{Institut Europlace de Finance}.

\bibliography{biblio.bib}

\begin{thebibliography}{21}
\providecommand{\natexlab}[1]{#1}
\providecommand{\url}[1]{\texttt{#1}}
\expandafter\ifx\csname urlstyle\endcsname\relax
  \providecommand{\doi}[1]{doi: #1}\else
  \providecommand{\doi}{doi: \begingroup \urlstyle{rm}\Url}\fi

\bibitem[Bardos et~al.(2000)Bardos, Golse, and Mauser]{bardos2000weak}
C.~Bardos, F.~Golse, and N.~J. Mauser.
\newblock Weak coupling limit of the $ n $-particle schr{\"o}dinger equation.
\newblock \emph{Methods and Applications of Analysis}, 7\penalty0 (2):\penalty0
  275--294, 2000.

\bibitem[Belavkin(1992)]{belavkin1992quantum}
V.~P. Belavkin.
\newblock Quantum stochastic calculus and quantum nonlinear filtering.
\newblock \emph{Journal of Multivariate analysis}, 42\penalty0 (2):\penalty0
  171--201, 1992.

\bibitem[Benedikter et~al.(2016)Benedikter, Porta, Schlein,
  et~al.]{benedikter2016effective}
N.~Benedikter, M.~Porta, B.~Schlein, et~al.
\newblock \emph{Effective evolution equations from quantum dynamics}, volume~7.
\newblock Springer, 2016.

\bibitem[Bouten et~al.(2007)Bouten, Van~Handel, and
  James]{bouten2007introduction}
L.~Bouten, R.~Van~Handel, and M.~R. James.
\newblock An introduction to quantum filtering.
\newblock \emph{SIAM Journal on Control and Optimization}, 46\penalty0
  (6):\penalty0 2199--2241, 2007.

\bibitem[Cazenave(2003)]{cazenave2003semilinear}
T.~Cazenave.
\newblock \emph{Semilinear Schrodinger Equations}, volume~10.
\newblock American Mathematical Soc., 2003.

\bibitem[Chalal et~al.(2023)Chalal, Amini, and Guo]{chalal2023mean}
S.~Chalal, N.~H. Amini, and G.~Guo.
\newblock On the mean-field belavkin filtering equation.
\newblock \emph{IEEE Control Systems Letters}, 7:\penalty0 2910--2915, 2023.

\bibitem[de~Bouard and Debussche(2003)]{debouard2003snls}
A.~de~Bouard and A.~Debussche.
\newblock The stochastic nonlinear {Schr{\"o}dinger} equation in {{\(H^1\)}}.
\newblock \emph{Stochastic Anal. Appl.}, 21\penalty0 (1):\penalty0 97--126,
  2003.
\newblock ISSN 0736-2994.

\bibitem[Erd{\"o}s and Yau(2001)]{erdos2001derivation}
L.~Erd{\"o}s and H.~Yau.
\newblock Derivation of the nonlinear schr{\"o}dinger equation from a many body
  coulomb system.
\newblock \emph{Advances in Theoretical and Mathematical Physics}, 5\penalty0
  (6), 2001.

\bibitem[Golse(2016)]{golse2016dynamics}
F.~Golse.
\newblock On the dynamics of large particle systems in the mean field limit.
\newblock In \emph{Macroscopic and large scale phenomena: coarse graining, mean
  field limits and ergodicity}, pages 1--144. Springer, 2016.

\bibitem[Hepp(1974)]{Hepp1974}
K.~Hepp.
\newblock The classical limit for quantum mechanical correlation functions.
\newblock \emph{Communications in Mathematical Physics}, 35\penalty0
  (4):\penalty0 265--277, 1974.
\newblock \doi{10.1007/BF01646348}.
\newblock URL \url{https://doi.org/10.1007/BF01646348}.

\bibitem[Knowles and Pickl(2010)]{knowles2010mean}
A.~Knowles and P.~Pickl.
\newblock Mean-field dynamics: singular potentials and rate of convergence.
\newblock \emph{Communications in Mathematical Physics}, 298:\penalty0
  101--138, 2010.

\bibitem[Kolokoltsov(2021)]{kolokoltsov2021law}
V.~N. Kolokoltsov.
\newblock The law of large numbers for quantum stochastic filtering and control
  of many-particle systems.
\newblock \emph{Theoretical and Mathematical Physics}, 208:\penalty0 937--957,
  2021.

\bibitem[Kolokoltsov(2022)]{kolokoltsov2022quantum}
V.~N. Kolokoltsov.
\newblock Quantum mean-field games.
\newblock \emph{The Annals of Applied Probability}, 32\penalty0 (3):\penalty0
  2254--2288, 2022.

\bibitem[Kolokoltsov(2025{\natexlab{a}})]{kolokoltsov2025mathematicaltheoryquantumstochastic}
V.~N. Kolokoltsov.
\newblock On the mathematical theory of quantum stochastic filtering equations
  for mixed states, 2025{\natexlab{a}}.
\newblock URL \url{https://arxiv.org/abs/2505.14605}.

\bibitem[Kolokoltsov(2025{\natexlab{b}})]{kolokoltsov2025quantum}
V.~N. Kolokoltsov.
\newblock On quantum stochastic master equations.
\newblock \emph{Electronic Journal of Probability}, 30:\penalty0 1--21,
  2025{\natexlab{b}}.

\bibitem[Kukush(2019)]{kukush2019}
A.~Kukush.
\newblock \emph{Gaussian measures in {Hilbert} space. {Construction} and
  properties}.
\newblock Math. Stat. Ser. Hoboken, NJ: John Wiley \& Sons; London: ISTE, 2019.
\newblock ISBN 978-1-78630-267-0; 978-1-119-47682-5.

\bibitem[Mora and Rebolledo(2008)]{mora2008basic}
C.~M. Mora and R.~Rebolledo.
\newblock Basic properties of nonlinear stochastic {Schr{\"o}dinger} equations
  driven by {Brownian} motions.
\newblock \emph{Ann. Appl. Probab.}, 18\penalty0 (2):\penalty0 591--619, 2008.
\newblock ISSN 1050-5164.

\bibitem[Pickl(2011)]{pickl2011simple}
P.~Pickl.
\newblock A simple derivation of mean field limits for quantum systems.
\newblock \emph{Letters in Mathematical Physics}, 97:\penalty0 151--164, 2011.

\bibitem[Rodnianski and Schlein(2009)]{rodnianski2009quantum}
I.~Rodnianski and B.~Schlein.
\newblock Quantum fluctuations and rate of convergence towards mean field
  dynamics.
\newblock \emph{Communications in Mathematical Physics}, 291\penalty0
  (1):\penalty0 31--61, 2009.

\bibitem[Rougerie(2021)]{rougerie2021scaling}
N.~Rougerie.
\newblock Scaling limits of bosonic ground states, from many-body to non-linear
  schr{\"o}dinger.
\newblock \emph{EMS Surveys in Mathematical Sciences}, 7\penalty0 (2):\penalty0
  253--408, 2021.

\bibitem[Spohn(1980)]{spohn1980kinetic}
H.~Spohn.
\newblock Kinetic equations from hamiltonian dynamics: Markovian limits.
\newblock \emph{Reviews of Modern Physics}, 52\penalty0 (3):\penalty0 569,
  1980.

\end{thebibliography}

\end{document}